\documentclass{amsart}

\usepackage[utf8]{inputenc}

\usepackage{geometry}
\usepackage{graphicx}
\usepackage{amssymb}
\usepackage{amsthm}
\usepackage{mathrsfs}

\usepackage{stackrel}
\usepackage[all]{xy}

\usepackage{color}

\usepackage{hyperref}
\hypersetup{
     colorlinks = true,
     linkcolor = magenta,
     anchorcolor = magenta,
     citecolor = magenta,
     filecolor = magenta,
     urlcolor = magenta
     }

\newcommand{\Z}{ {\mathbf{Z}} } 
\newcommand{\Q}{ {\mathbf{Q}} }
\newcommand{\R}{ {\mathbf{R}} }
\renewcommand{\P}{ {\mathbf{P}} }

\newcommand{\C}{ {\mathbf{C}} }
\newcommand{\Zp}{ {\Z_p} }

\newcommand{\barQp}{ {\bar{\mathbf{Q}}_p} }

\newcommand{\cO}{ {\mathcal{O}} }

\newcommand{\cU}{{\mathcal{U}}}

\usepackage{yfonts}
\renewcommand{\d}{\mathfrak d}
\newcommand{\dd}{\mathfrak d'}

\newcommand{\Sym}{{ \mathrm{Sym}}}

\newcommand{\GL}{\mathrm{GL}}

\newcommand{\SL}{\mathrm{SL}}

\newcommand{\ord}{\mathrm{ord}}

\newcommand{\cont}{\mathrm{cont}}

\newcommand{\Hom}{{\mathrm{Hom}}}

\newcommand{\Div}{\mathrm{Div}}

\newcommand{\Symb}{{ \mathrm{Symb}}}

\newcommand{\hf}{ {\mathbf{f}} }

 \newcommand{\comm}[1]{}

\newtheorem{theorem}{Theorem}[section]
\newtheorem{definition}[theorem]{Definition}
\newtheorem{lemma}[theorem]{Lemma}

\newtheorem{remark}[theorem]{Remark}
\newtheorem{corollary}[theorem]{Corollary}

\newtheorem{proposition}[theorem]{Proposition}

\newtheorem*{theorem*}{Theorem}

\title{$p$-adic families of $\d$-th Shintani liftings}

\author{Daniele Casazza}
\address{University College Dublin. Belfield, Dublin 4, Ireland.}
\email{casazza.daniele@gmail.com}

\author{Carlos de Vera-Piquero}
\address{Universidad Politécnica de Madrid. ETSIT, Avda. Complutense 30, 28040 Madrid, Spain.}
\email{cdeverapiquero@gmail.com}

\date{2020.}

\begin{document}

\begin{abstract}
In this note we give a detailed construction of a $\Lambda$-adic $\d$-th Shintani lifting. We obtain a $\Lambda$-adic version of Kohnen's formula relating Fourier coefficients of half-integral weight modular forms and special values of twisted $L$-series. As a by-product, we derive a mild generalization of such classical formulae, and also point out a relation between Fourier coefficients of $\Lambda$-adic $\d$-th Shintani liftings and Stark--Heegner points.
\end{abstract}


\maketitle

\setcounter{tocdepth}{1}
\tableofcontents

\section{Introduction}

In his seminal paper \cite{Shi73}, Shimura unrevealed the first instance of what is nowadays called theta (or Howe) correspondence, which was systematically studied later by Waldspurger \cite{Waldspurger80, Waldspurger91}. Shimura's work contained an in-depth study of half-integral weight modular forms, and provided Hecke-equivariant linear maps
\[
    \mathcal S_{k,N,\chi,\d}\colon S_{k+1/2}^+(N,\chi) \, \longrightarrow \, S_{2k}(N,\chi^2)
\]
from Kohnen's plus subspace of the space of half-integral weight modular forms to the space of classical modular forms of even weight. The construction depends on an auxiliary discriminant $\d\in\Z$ and different choices yield different maps.

By means of certain cycle integrals along geodesic paths on the complex upper half plane, one can define Hecke-equivariant linear maps 
\[
\theta_{k,N,\chi,\d}\colon S_{2k}(N,\chi^2) \, \longrightarrow \, S_{k+2/1}^+(N,\chi)
\]
which are adjoint to $\mathcal S_{k,N,\chi,\d}$ with respect to the Petersson product, meaning that
\[
\langle g, \theta_{k,N,\chi,\d}(f) \rangle = \langle \mathcal S_{k,N,\chi,\d}(g), f \rangle \quad \text{for all } f \in S_{2k}(N,\chi^2), g \in S_{k+1/2}^+(N,\chi).
\]
This construction was first studied by Shintani \cite{Shintani}, and subsequently extended by Kohnen and Zagier \cite{KZ81}, Kohnen \cite{Kohnen-Newforms, Kohnen-FourierCoefficients}, and Kojima--Tokuno \cite{KojimaTokuno}, among others. The maps $\theta_{k,N,\chi,\d}$ are referred to as {\em $\d$-th Shintani liftings}. We will drop $\chi$ from the notation when it is trivial.

Under certain assumptions, for example when $N$ is squarefree and $\chi$ is trivial, a theory of newforms of half-integral weight {\em à la} Atkin--Li--Miyake is available, and the $\d$-th Shimura and $\d$-th Shintani liftings can be used to establish a Hecke-equivariant isomorphism between $S_{2k}^{new}(N,\chi^2)$ and $S_{k+1/2}^{+,new}(N,\chi)$: this is the so-called {\em Shimura--Shintani correspondence}. In this direction, one of the main motivations of the work of Kohnen and Zagier was to obtain an explicit Waldspurger-type formula relating Fourier coefficients of half-integral weight modular forms and twisted $L$-values of classical modular forms. For instance, suppose that $N$ is odd and squarefree, and that $g\in S_{k+1/2}^{+,new}(N)$ and $f \in S_{2k}^{new}(N)$ are two non-zero new modular forms in Shimura--Shintani correspondence. Then, Kohnen's formula \cite[Corollary 1]{Kohnen-FourierCoefficients} asserts that for any fundamental discriminant $D$ with $(-1)^kD>0$ and such that $(\frac{D}{\ell})=w_{\ell}$ for all primes $\ell \mid D$, where $w_{\ell}$ is the eigenvalue of the Atkin--Lehner involution $W_{\ell}$ acting on $f$, one has 
\begin{equation}\label{KohnenFormula-intro}
    \frac{|a_{|D|}(g)|^2}{\langle g,g\rangle} = 2^{\nu(N)} \frac{(k-1)!}{\pi^k}|D|^{k-1/2} \frac{L(f,D,k)}{\langle f,f \rangle }.
\end{equation}
Here, $a_{|D|}(g)$ denotes the $|D|$-th Fourier coefficient of $g$, and $\nu(N)$ is the number of prime divisors of $N$. We insist that fixed a newform $f \in S_{2k}^{new}(N)$, this formula is valid for {\em any} $g \in S_{k+1/2}^{+,new}(N)$ in Shimura--Shintani correspondence with $f$, since any two such forms will lie on the same line.

One of the key ingredients in the proof of \eqref{KohnenFormula-intro} is a formula that relates directly the $|\d|$-th Fourier coefficient of the $\d$-th Shintani lifting of $f$ with the twisted special value $L(f,\d,k)$. Indeed, assuming that $(-1)^k\d > 0$ and that $(\frac{\d}{\ell})=w_{\ell}$ for all primes $\ell \mid N$, it is shown in \cite{Kohnen-FourierCoefficients} that
\begin{equation}\label{aD-coefficient}
a_{|\d|}(\theta_{k,N,\d}(f)) = (-1)^{[k/2]}2^{\nu(N)+k} |\d|^k(k-1)! \cdot \frac{L(f,\d,k)}{(2\pi i)^k \mathfrak g(\chi_{\d})},
\end{equation}
where $\mathfrak g(\chi_{\d})$ is the Gauss sum attached to the quadratic character $\chi_{\d}$. While Kohnen's formula in \eqref{KohnenFormula-intro} depends notably on having a good theory of newforms as cited above (which in particular provides `multiplicity one'), and therefore does not extend easily when dropping the assumptions that $N$ is squarefree and $\chi$ is trivial, the formula in \eqref{aD-coefficient} does generalize quite easily. We refer the reader to Kojima--Tokuno \cite{KojimaTokuno} for an extension of Kohnen's work and ideas.

The pioneering work of Shimura and Waldspurger not only motivated the above mentioned works by Kohnen and Zagier, but it has also inspired many other investigations over the course of several years. One of the most celebrated ones was a formula of Gross--Kohnen--Zagier \cite{GKZ}, relating Fourier coefficients of half-integral weight modular forms (or of Jacobi forms) to Heegner divisors. More in line with the $p$-adic flavour of our work here, Darmon--Tornar\'ia \cite{DarmonTornaria} proved a Gross--Kohnen--Zagier type formula for Stark--Heegner points attached to real quadratic fields. Their formula led to a similar relationship as in Kohnen's formula for central critical {\em derivatives}, with the role of the Fourier coefficient $a_{|D|}(g)$ being played by the first derivative of the $|D|$-th Fourier coefficient of a {\em $p$-adic family} of half-integral forms. Also in this line, the $p$-adic variation of the Gross--Kohnen--Zagier theorem, including the existence of $\Lambda$-adic families of Jacobi forms, is studied in \cite{LN19a,LN19b}.

In this paper we focus on the interpolation of the above liftings in $p$-adic families. Our main source of inspiration is the work of Stevens \cite{Stevens} on the $p$-adic interpolation of the {\em first} (i.e. $\d=1$) Shintani lifting, complementing Hida \cite{Hid95} towards the understanding of a $\Lambda$-adic Shimura--Shintani correspondence. Stevens's strategy combined a cohomological interpretation of Shintani's cycle integrals with the theory of $\Lambda$-adic modular symbols developed in \cite{GS93}. 
 
The primary goal of this note is to describe a $\Lambda$-adic $\d$-th Shintani lifting interpolating $p$-adically the classical $\d$-th Shintani lifting when $\d>1$. The main novelty of our approach is that we can derive a $\Lambda$-adic version of Kohnen's formula stated in \eqref{aD-coefficient}. This formula, which will be further described below, relates the $|\d|$-th Fourier coefficient of the $\Lambda$-adic $\d$-th Shintani lifting to a suitable $p$-adic $L$-function interpolating twisted central $L$-values of $p$-ordinary eigenforms (built from the two-variable $p$-adic $L$-function of Greenberg--Stevens). The reader may interpret this formula as a cuspidal analogue of the well-known relation between the $0$-th Fourier coefficient of a $\Lambda$-adic Eisenstein series and the relevant Kubota--Leopoldt $p$-adic $L$-function.

In more detail, given a Hida family $\mathbf f$ of ordinary $p$-stabilized newforms of tame level $N$ and tame character $\chi^2$, we construct a $p$-adic family of half-integral weight modular forms interpolating the $\d$-th Shintani liftings of the classical specializations of $\mathbf f$. To describe our main results, suppose that $\mathbf f$ is given by a power series $\mathbf f \in \mathcal R[[q]]$, where $\mathcal R$ is a finite flat integral extension of the Iwasawa algebra $\Lambda = \Z_p[[1+p\Z_p]] \simeq \Z_p[[T]]$. Let $\mathcal U^{\mathrm{cl}}$ denote the dense subset of classical points of $\mathcal U$, the open in the $p$-adic weight space $\mathcal W = \Hom(\mathcal R, \bar\Q_p)$ where the Hida family is defined, i.e. $\mathbf f(\kappa) \in S_{2k}(Np,\chi^2)$ is the $q$-expansion of an ordinary $p$-stabilized newform of level $Np$ and character $\chi^2$ for all $\kappa \in \mathcal U^{\mathrm{cl}}$. We show in Theorem \ref{thm:LambdaShintani-interpolation} that the $\d$-th Shintani liftings of the classical specializatons $\mathbf f(\kappa)$ are $p$-adically interpolated by a power series $\Theta_{\d}(\mathbf f) \in \widetilde{\mathcal R}[[q]]$, where $\widetilde{\mathcal R} \to \mathcal R$ is the {\em metaplectic covering} of $\mathcal R$ (cf. Section \ref{Sec:Lambda-adic}). This covering induces a natural map $\pi\colon \widetilde{\mathcal W} \to \mathcal W$ on $p$-adic weight spaces that `doubles' the weight of the cassical points.

\begin{theorem*}
    Let $\d$ be a fundamental discriminant with $\d \equiv 0 \pmod p$. There exists a power series 
    \[
        \Theta_{\d}(\mathbf f) = \sum_{m\geq 1} \mathbf a_m(\Theta_{\d}(\mathbf f)) q^m \in \widetilde{\mathcal R}[[q]]
    \]
defined on a non-empty open subset $\widetilde{\mathcal U} \subset  \widetilde{\mathcal W}$,
uniquely determined by the fact that:
\[
    \Theta_{\d}(\mathbf f)(\tilde{\kappa}) = C(k,\chi,\d)^{-1} \frac{\Omega_{\kappa}}{\Omega_{\mathbf f(\kappa)}^-}  \theta_{k,Np,\chi,\d}(\mathbf f(\kappa)), \qquad \forall \tilde{\kappa} \in \widetilde{\mathcal U}^{\mathrm{cl}},
\]
where $\kappa = \pi(\tilde{\kappa}) \in \mathcal U^{\mathrm{cl}}$. Here, $C(k,\chi,\d)$ is a constant defined in \eqref{constant:Ckchid}, and $\Omega_{\mathbf f(\kappa)}$ and $\Omega_{\kappa}$ are complex and $p$-adic periods attached to $\mathbf f(\kappa)$ and $\kappa$, respectively (cf. Section \ref{sec:padicMS}).
\end{theorem*}

We refer the reader to Section \ref{Sec:Lambda-adic} for the precise statement and a detailed proof. The construction of $\Theta_{\d}(\mathbf f)$ is based on a cohomological interpretation of the $\d$-th Shintani lifting by using modular symbols and the $p$-adic interpolation of the latter. We must point out that one can construct a priori {\em two} different $\Lambda$-adic $\d$-th Shintani liftings of a given Hida family $\mathbf f$, each of them satisfying an interpolation property on a different subset of classical points. 
The meaning of this dichotomy and the choice of $\widetilde{\mathcal U}$ are explained in Section \ref{sec:LambdaShintani}.

As a consequence of the existence of a $\Lambda$-adic $\d$-th Shintani lifting, we prove in Theorem \ref{thm:comparisonGS} a $\Lambda$-adic version of Kohnen's formula \eqref{aD-coefficient}. In the simplified case in which $\chi$ is the trivial character Theorem \ref{thm:comparisonGS} reduces to the following identity in $\widetilde{\mathcal R}$:
\[
    \mathbf a_{|\d|}(\Theta_{\d}(\mathbf f)) = \mathrm{sgn}(\d) \cdot 2^{\nu(N)} \cdot \tilde{a}_p(\mathbf f) \cdot  \tilde{\mathcal L}_p^{\mathrm{GS}}(\mathbf f, \d).
\]
Here, the $p$-adic $L$-function $\tilde{ \mathcal L}_p^{\mathrm{GS}}(\mathbf f, \d) \in \widetilde{\mathcal R}$ interpolates the values $L(\mathbf f(\kappa), \d, k)$ and it is built from a suitable restriction of the two-variable $p$-adic $L$-function studied by Greenberg--Stevens in \cite{GS93}, and $\tilde{a}_p(\mathbf f)$ is the pullback of $a_p(\mathbf f)$ along the map $\pi:\widetilde{\mathcal W} \to \mathcal W$. As a by-product, we obtain a mild generalization of the classical formula in \eqref{aD-coefficient} (see Corollary \ref{cor:KohnenFormula-Np}), and we also describe exceptional zero phenomena for the coefficients $\mathbf a_{|\d|}(\Theta_{\d}(\mathbf f))$ (see Section \ref{sec:Shadows}). Using a result of Bertolini--Darmon \cite{BD07}, in a particular exceptional zero setting the second derivative of $ \mathbf a_{|\d|}(\Theta_{\d}(\mathbf f))$ can be related to the logarithm of a (Stark--)Heegner point (cf. Corollary \ref{cor:BDderivative} below). Similar results for the first derivative were studied in \cite{DarmonTornaria} and ours may be seen as complementary. This leads us to believe that our $\Lambda$-adic $\d$-th Shintani lifting can push further the study of these applications, by providing a robust and systematic approach to a $\Lambda$-adic Shimura--Shintani correspondence.

The $\Lambda$-adic $\d$-th Shintani lifting constructed in this note is a key ingredient in the definition of a certain $p$-adic $L$-function for $\GL_2 \times \GL_3$ associated with two Hida families of ordinary cusp forms that was hinted at in \cite{PaldVP, PaldVP2}. This is studied in a forthcoming work (see \cite{CdVP}). 

We should mention a few works that are in the orbit of this article. First of all, we must say that Stevens' $\Lambda$-adic version of the {\em first} Shintani lifting was generalized by Park \cite{Park} to the non-ordinary finite slope case. In contrast to this, we stay in the ordinary case but extend Stevens' ideas to the case where $\d > 1$ is an arbitrary fundamental discriminant. Secondly, it is crucial in our construction the discussion in Section \ref{Sec:p-stabilizations} on the $\d$-th Shintani lifting of $p$-stabilized newforms, done by means of a careful study of certain sets of integral binary quadratic forms. This approach resembles the discussion in the work of Makiyama in \cite{Mak17}, who does the study when $p\nmid \d$. However, we believe there is a gap in Makiyama's computations, which can be circumvented precisely when $p$ divides $\d$ with our detailed analysis in Section \ref{sec:discussion-quadforms}. Our setting is therefore complementary, and along the course of our discussion we also give a precise relation between the $\d$-th Shintani lifting of a newform and that of its $p$-stabilization (not only for the $|\d|$-th Fourier coefficient). Finally, we should also mention the preprint \cite{Kaw} by Kawamura, who constructs certain $p$-adic families of Siegel cusp forms of arbitrary genus interpolating Duke--Imamo\={g}lu liftings of level $N=1$. We must point out that our construction provides a much more general version of the $\Lambda$-adic $\d$-th Shintani lifting than the one needed in Kawamura's discussion.

To close this introduction, let us briefly explain the organization of the article. In Section \ref{Sec:Shintani-Lifting} we set the basic definitions for integral binary quadratic forms, we describe Kojima--Tokuno's generalization of Kohnen's $\d$-th Shintani lifting, and we also explain the exact relation with special values of $L$-functions. Section \ref{Sec:p-stabilizations} is devoted to a classification result for integral binary quadratic forms from \cite{GKZ}, which we use to derive the exact relations between the liftings in level $N$ and $Np$, leading to a comparison between the Fourier coefficients of the lifting of a modular form and those of its $p$-stabilization. In Section \ref{Sec:p-adic} we recall some basics of Hida theory and of the $p$-adic interpolation of modular symbols {\em \`a la} Greenberg--Stevens. Section \ref{Sec:Lambda-adic} contains the main results and applications. In particular, the definition of $\Theta_{\d}$ is given in equation \eqref{def:LambdaShintani}, the interpolation property is proved in Theorem \ref{thm:LambdaShintani-interpolation}, and the above mentioned $\Lambda$-adic Kohnen formula is proved in Theorem \ref{thm:comparisonGS}.

\vspace{0.2cm} 

\noindent {\bf Acknowledgements.} D.C. acknowledges financial support by Spanish Ministry of Economy and Competitiveness, through the Severo Ochoa programme for Centres of Excellence in R\&D (SEV-2015-0554), and C.dVP. acknowledges financial support by a Junior Researcher Grant through the project AGAUR PDJ2012. Both authors are also grateful to ICMAT and UB for their warm hospitality, and to Prof. Dieulefait for partial support through the project MTM2015-66716-P. This project has also received funding from the European Research Council under the European Union’s Horizon 2020 research and innovation programme (grant agreement No 682152).

\vspace{0.2cm}

\subsection{Notation}\label{sec:notation}

We shall fix the following general notation throughout the entire paper. As usual $\Z$, $\Q$, $\R$, and $\C$ will denote the ring of integers, the field of rational numbers, the field of real numbers, and the field of complex numbers, respectively. If $z\in \C^{\times}$ and $x \in \C$, we define $z^x = e^{x\log z}$, where $\log z = \log|z|+i\arg(z)$ with $-\pi < \arg(z) \leq \pi$. If $\psi$ is a  Dirichlet character of conductor $c$, we write its Gauss sum as
\[
\mathfrak g(\psi) = \sum_{a \in (\Z/c\Z)^{\times}} \psi(a) e^{2\pi i a/c}.
\]
When $\psi$ is primitive, one has $|\mathfrak g(\psi)|^2 = \psi(-1)\mathfrak g(\psi)\mathfrak g(\bar{\psi}) = c$. If $\psi$, $\psi'$ are primitive Dirichlet characters of relatively prime conductors $c$, $c'$, respectively, then $\mathfrak g(\psi\psi') = \psi(c')\psi'(c) \mathfrak g(\psi)\mathfrak g(\psi')$.

For any commutative ring $R$ with unit, $\SL_2(R)$ will denote the special linear group over $R$. The group $\SL_2(\R)$ acts as usual on the complex upper half plane $\mathfrak H$ via linear fractional transformations. If $r, M \geq 1$ are integers, and $\psi$ is a Dirichlet character modulo $M$, we write $S_r(M,\psi)$ for the (complex) space of cusp forms of weight $r$, level $\Gamma_0(M)$ and character $\psi$. We define the Petersson product of two cusp forms $f, g \in S_r(M,\psi)$ by 
\[
\langle f, g \rangle = \frac{1}{i_{M}}\int_{\Gamma_0(M)\backslash \mathfrak H} f(z)\overline{g(z)}y^{r-2} dx dy, 
\]
where $z = x+iy$ and $i_M = [\SL_2(\Z):\Gamma_0(M)]$. With this normalization, the Petersson product of $f$ and $g$ does not change if we replace $M$ by a multiple of it and see $f$ and $g$ as forms of that level. 

If $N\geq 1$ is an {\em odd} integer, $k \geq 0$ is an integer, and $\chi$ is a Dirichlet character modulo $N$, we write $\tilde{\chi}$ for the Dirichlet character modulo $4N$ given by $(\frac{4\epsilon}{\cdot})\chi$, where $\epsilon = \chi(-1)$, and write $S_{k+1/2}(4N,\tilde{\chi})$ for the (complex) space of cusp forms of half-integral weight $k+1/2$, level $4N$ and character $\tilde{\chi}$, in the sense of Shimura \cite{Shi73}. Observe that $\tilde{\chi}$ is an even character by construction. If $f, g \in S_{k+1/2}(4N,\tilde{\chi})$, their Petersson product is \[
\langle f, g \rangle = \frac{1}{i_{4N}}\int_{\Gamma_0(4N)\backslash \mathfrak H} f(z)\overline{g(z)}y^{k-3/2} dx dy. 
\]

We will denote by $S_{k+1/2}^+(N,\chi)$ the subspace of $S_{k+1/2}(4N,\tilde{\chi})$ usually referred to as `Kohnen's plus space', consisting of those forms $f$ whose $q$-expansion has the form 
\[
f(z) = \sum_{\substack{n\geq 1,\\ \epsilon (-1)^k n \equiv 0,1 \, (4)}} a(n) q^n.
\]

We may recall the definition of Hecke operators acting on the space $S_{k+1/2}^+(N,\chi)$. For a prime $p \nmid N$, and $f = \sum a(n)q^n \in S_{k+1/2}^+(N,\chi)$, the action of the Hecke operator $T_{k+1/2,N,\chi}(p^2)$ is given by Theorem 1.7 of \cite{Shi73} with the following formula:
\begin{equation} \label{eqn:Tp-Shimura}
T_{k+1/2,N,\chi}(p^2)f(z) = \hspace{-15pt} \sum_{\substack{n\geq 1,\\ \epsilon(-1)^k\equiv 0,1(4)}}  \hspace{-15pt}  \Bigl(a(p^2n) + \chi(p)\left(\frac{\epsilon (-1)^k n}{p}\right)p^{k-1}a(n)+\chi(p^2)p^{2k-1}a(n/p^2)\Bigr)q^n,
\end{equation}
where one reads $a(n/p^2)= 0$ if $p^2\nmid n$. For primes $p \mid N$, one also defines operators $U(p^2)$ by setting
\[
U(p^2)f(z) =  \sum_{\substack{n\geq 1,\\ \epsilon(-1)^k\equiv 0,1(4)}} a(p^2n)q^n.
\]

\section{\texorpdfstring{$\d$}{d}-th Shintani lifting} \label{Sec:Shintani-Lifting}

The aim of this section is to review the definition of the so-called $\d$-th Shintani lifting, where $\d\in\Z$ is a fixed fundamental discriminant (i.e., either $\d=1$ or $\d$ is the discriminant of a quadratic field). This is a Hecke-equivariant linear map from integral weight modular forms to half-integral weight modular forms, studied in detail by Kohnen in \cite{Kohnen-FourierCoefficients}, building on the seminal work of Shintani \cite{Shintani}, and later generalized by others. We will briefly summarize the approach in Kojima--Tokuno \cite{KojimaTokuno}, which essentially adapts Kohnen's work for arbitrary odd level and arbitrary nebentype character.

Before entering into the proper construction of the $\d$-th Shintani lifting, we fix some notations concerning integral binary quadratic forms that will be of good use.

\subsection{Integral binary quadratic forms}

We write $\mathcal Q$ for the set of all integral binary quadratic forms 
\[
[a,b,c](X,Y) = aX^2+bXY +cY^2, \quad a, b, c \in \Z,
\]
on which $\Gamma_0(1) = \SL_2(\Z)$ acts by the rule 
\[
([a,b,c]\circ\gamma)(X,Y) = [a,b,c]((X,Y) {}^t \gamma), \quad \gamma \in \SL_2(\Z).
\]
If one identifies the quadratic form $Q = [a,b,c]$ with the symmetric matrix 
\[
A_Q = \left(\begin{array}{cc} a & b/2 \\ b/2 & c \end{array} \right),
\]
then $[a,b,c]\circ \gamma$ corresponds to the matrix ${}^t\gamma A_Q \gamma$. Given $Q = [a,b,c] \in \mathcal Q$, its discriminant is by definition $b^2-4ac=\det(2A_Q)$. It is immediate that the discriminant is invariant under the above $\Gamma_0(1)$-action. 

If $\Delta$ is a discriminant, we write $\mathcal Q(\Delta)$ for the subset of quadratic forms in $\mathcal Q$ having discriminant $\Delta$. There is an induced $\Gamma_0(1)$-action on $\mathcal Q(\Delta)$. If $\d$ is a fundamental discriminant dividing $\Delta$ and $Q =[a,b,c] \in \mathcal Q(\Delta)$, then we set 
\[
\omega_{\d}(Q) = \begin{cases}
0 & \text{if } \mathrm{gcd}(a,b,c,\d) > 1, \\
\left(\frac{\d}{r}\right) & \text{if } \mathrm{gcd}(a,b,c,\d)=1 \text{ and } Q \text{ represents } r, \, \mathrm{gcd}(r,\d)=1.
\end{cases}
\]
One can easily check that this definition does not depend on the choice of the integer $r$, when $\mathrm{gcd}(a,b,c,\d) = 1$. In addition, the value of $\omega_{\d}(Q)$ depends only on the $\Gamma_0(1)$-equivalence class of $Q$. Besides the definition, when $\Delta > 0$ one can compute $\omega_{\d}(Q)$ by using the following explicit formula (cf. \cite[p.263, Proposition 6]{Kohnen-FourierCoefficients})
\begin{equation}\label{genuschar-formula}
\omega_{\d}([a,b,c]) = \prod_{q^{\nu} || a} \left(\frac{\d/q^*}{q^{\nu}}\right)\left(\frac{q^*}{ac/q^{\nu}}\right).
\end{equation} 
Here, $q$ runs over the prime factors of $a$, $q^{\nu}||a$ means that $q^{\nu}\mid a$ and $\mathrm{gcd}(q,a/q^{\nu})=1$, and $q^* := \left(\frac{-1}{q}\right)q$.

We denote by $\mathcal Q^1(\Delta)$ the subset of {\em primitive} forms in $\mathcal Q(\Delta)$, namely those forms for which $\mathrm{gcd}(a,b,c)=1$. The induced function $\omega_{\d}\colon \mathcal Q^1(\Delta)/\Gamma_0(1) \to \{\pm 1\}$ is usually referred to as a {\em genus character}. As explained in \cite[Ch. I]{GKZ}, $\mathcal Q^1(\Delta)/\Gamma_0(1)$ has a natural group structure and $\omega_{\d}$ is a group homomorphism; and conversely, every group homomorphism $\mathcal Q^1(\Delta)/\Gamma_0(1) \to \{\pm 1\}$ is of the form $\omega_{\dd}$ for some fundamental discriminant $\dd$ dividing $\Delta$, the only relations being that $\omega_{\d} = \omega_{\dd}$ if $\Delta=\d\dd m^2$ for some natural number $m$.

If $M\geq 1$ is an integer, we also denote by $\mathcal Q_M(\Delta)$ the subset of forms $Q=[a,b,c] \in \mathcal Q(\Delta)$ such that $a \equiv 0 \pmod M$. One can easily check that the congruence subgroup $\Gamma_0(M)$ acts on $\mathcal Q_M(\Delta)$. If $t > 0$ is a divisor of $M$, then the map 
\begin{equation}\label{bijection-t}
Q = [a,b,c] \mapsto Q_t := Q \circ \left(\begin{array}{cc} t & 0 \\ 0 & 1 \end{array}\right) = [at^2, bt, c]
\end{equation}
yields a bijection $\mathcal Q_{M/t}(\Delta) \to \mathcal Q_{Mt}(\Delta t^2)$. If $\d$ is a fundamental discriminant dividing $\Delta$ as above, one has
\begin{equation}\label{genuschar-relation-t}
\omega_{\d}(Q_t) = \chi_{\d}(t^2)\omega_{\d}(Q),
\end{equation}
which equals $\omega_{\d}(Q)$ if $\mathrm{gcd}(t,\d) = 1$. When we do not want to specify the discriminant, $\mathcal Q_M$ will denote the union of all the sets $\mathcal Q_M(\Delta)$.

\subsection{The $\d$-th Shimura and Shintani liftings}

Fix through all this paragraph an odd integer $N\geq 1$ and an integer $k \geq 1$, and fix also a Dirichlet character $\chi$ modulo $N$. Let $N_0 \geq 1$ be the conductor of $\chi$, $\chi_0$ be the primitive character modulo $N_0$ associated with $\chi$, and put $N_1 = N/N_0$ and $\epsilon = \chi(-1)$. Although it is not always needed, for simplicity we assume through all our discussion that $\mathrm{gcd}(N_0,N_1) = 1$. Also, for technical reasons that will be apparent below, we make the following hypothesis.
\begin{equation}\label{hypothesisk=1}\tag{$\star$}
\text{if } k = 1, \text{ either } N \text{ is squarefree or } \chi \text{ is trivial and } N \text{ is cubefree}.
\end{equation}
Fix also a fundamental discriminant $\d$ satisfying $\mathrm{gcd}(N_0,\d)=1$. 

If $Q = [a,b,c] \in \mathcal Q$ is an integral binary quadratic form, we set $\chi_0(Q) := \chi_0(c)$. For an integer $u\geq 1$ such that $\mathrm{gcd}(N_0,u)=1$, and $\dd$ a discriminant with $\d\dd > 0$, we consider the set of integral binary quadratic forms 
\[
\mathcal Q_{N_0u}(N_0^2\d\dd) = \{ Q =[a,b,c] \in \mathcal Q: b^2-4ac = N_0^2\d\dd, \, N_0u \mid a \},
\]
as defined in the previous paragraph. Recall that there is a natural action of $\Gamma_0(N_0u)$ on this set. We consider now the subset
\[
\mathcal L_{N_0u}(N_0^2\d\dd) := \{Q = [a,b,c] \in \mathcal Q_{N_0u}(N_0^2\d\dd): \mathrm{gcd}(N_0,c)=1\},
\]
which is also invariant under  $\Gamma_0(N_0u)$. For $k\geq 2$, define a function  $f_{N_0,u}^{k,\chi_0}(z;\d,\d')$ of $z \in \mathfrak H$ by  
\begin{equation}\label{def-fkN0}
f_{N_0,u}^{k,\chi_0}(z;\d,\d') := \sum_{Q \in \mathcal L_{N_0u}(N_0^2\d\d')} \chi_0(Q)\omega_{\d}(Q)Q(z,1)^{-k}.
\end{equation}
These functions converge absolutely uniformly on compact sets, and from their definition one can easily check the following properties (cf. \cite[Section 1]{KojimaTokuno}):
\begin{itemize}
    \item[i)] $f_{N_0,u}^{k,\chi_0}(g\cdot z; \d, \dd) = \bar{\chi}_0^2(\delta)(\gamma z + \delta)^{2k} f_{N_0,u}^{k,\chi_0}(z; \d, \dd)$ for all 
    \[
    g = \left(\begin{array}{cc} \alpha & \beta \\ \gamma & \delta \end{array}\right) \in \Gamma_0(N_0u);
    \]
    \item[ii)] $f_{N_0,u}^{k,\chi_0}(-z; \d, \dd) = f_{N_0,u}^{k,\chi_0}(z; \d, \dd)$ (the map $[a,b,c]\mapsto [a,-b,c]$ gives a bijection of $\mathcal L_{N_0u}(N_0^2\d\dd)$ onto itself);
    \item[iii)] $f_{N_0,u}^{k,\chi_0}(\bar{z}; \d, \dd) = \overline{f_{N_0,u}^{k,\bar{\chi}_0}(z; \d, \dd)}$.
\end{itemize}
The functions $f_{N_0,u}^{k,\chi_0}(z; \d, \dd)$ yield cusp forms of weight $2k$, level $N_0u$, and character $\bar{\chi}_0^2$. For $k=1$, the series in \eqref{def-fkN0} is not absolutely convergent. However, one can apply ``Hecke's convergence trick'' to define $f_{N_0,u}^{k,\chi_0}(z; \d, \dd)$ in a similar manner (cf. \cite[p. 239]{Kohnen-FourierCoefficients}). In this case, hypothesis \eqref{hypothesisk=1} ensures that these functions are cusp forms as well. An explicit description of their Fourier coefficients (for $k\geq 2$) can be found in \cite[Proposition 1.2]{KojimaTokuno}.

\begin{remark}
The functions $f_{N_0,u}^{k,\chi_0}(z; \d, \dd)$ as above coincide with those denoted by $f_{k,N_0^2,u}(z; \d, \dd, \chi_0)$ in \cite{KojimaTokuno}. Indeed, the sum in equation \eqref{def-fkN0} could be taken over the sets $\mathcal Q_{N_0u}(N_0^2\d\dd)$ and even $\mathcal Q_{N_0^2u}(N_0^2\d\dd)$ remaining unchanged, due to the presence of the term $\chi_0(Q)$.
\end{remark}

Next consider the `kernel function' $\Omega_{k,N,\chi}(z,\tau; \d)$ of $(z,\tau) \in \mathfrak H \times \mathfrak H$ defined by
\[
    \Omega_{k,N,\chi}(z,\tau; \d) = i_N c_{k,\d,\chi}^{-1}  \hspace{-20pt} \sum_{\substack{m\geq 1, \\ \epsilon (-1)^k m \equiv 0,1 (4)}}
    \hspace{-15pt}
    m^{k-1/2}\left( \sum_{t \mid N_1} \mu(t)\chi_{\d}\bar{\chi}_0(t)t^{k-1}f_{N_0,N_1/t}^{k,\chi_0}(tz; \d, \epsilon(-1)^km)\right) e^{2\pi i m \tau},
\]
where 
\[
    i_N = [\Gamma_0(1): \Gamma_0(N)], 
    \quad 
    c_{k,\d,\chi} = (-1)^{[k/2]}|\d|^{-k+1/2}\pi \binom{2k-2}{k-1}2^{-3k+2} \epsilon^{k-1/2}N_0^{1-k}\frac{\mathfrak g(\chi_{\d})}{\mathfrak g(\chi_{\d}\bar{\chi}_0)}.
\]

For a fixed $\tau \in \mathfrak H$, the function $\Omega_{k,N,\chi}(\cdot , \tau; \d)$ on $\mathfrak H$ is a cusp form of weight $2k$, level $N$ and character $\bar{\chi}_0^2$ (for $k=1$, one needs again hypothesis \eqref{hypothesisk=1}), and it can be expressed in terms of certain Poincar\'e series by the identity in \cite[Theorem 2.2]{KojimaTokuno}, which is useful for computations. 

With our running assumptions on $N$, $k$, $\chi$ and $\d$, for each cusp form 
\[
    g(\tau) = \sum_{\substack{n\geq 1,\\\epsilon(-1)^k n \equiv 0,1 (4)}} c(n) e^{2\pi i n \tau} \in S_{k+1/2}^+(N,\chi)
\]
in Kohnen's plus space, one can define a function
\[
\mathcal S_{k,N,\chi,\d}(g) (z) = \sum_{n\geq 1} \left(\sum_{d\mid n} \chi_{\d}\chi(d) d^{k-1} c(n^2|\d|/d^2)\right) e^{2\pi i n z}.
\]
which satisfies the following property:
\[
\mathcal S_{k,N,\chi,\d}(g)(z) = \langle g, \Omega_{k,N,\chi}(-\bar z, \cdot ; \d)\rangle.
\]
It follows that, for a fixed $\tau$,  $z \mapsto \overline{\Omega_{k,N,\chi}(-\bar z, \tau; \d)}$ defines a cusp form of weight $2k$, level $N$, and character $\chi^2$. As a consequence, $g \mapsto \mathcal S_{k,N,\chi,\d}(g)$ yields a linear map 
\[
    \mathcal S_{k,N,\chi,\d}\colon S_{k+1/2}^+(N,\chi) \, \longrightarrow \, S_{2k}(N,\chi^2)
\]
with kernel function $\Omega_{k,N,\chi}(-\bar z, \cdot ; \d)$. In addition, this map commutes with Hecke operators (meaning that $T_{k+1/2,N,\chi}(p^2)$ corresponds to $T_{2k,N,\chi^2}(p)$ for $p\nmid N$ and $U(q^2)$ corresponds to $T_{2k,N,\chi^2}(q)$ for $q\mid N$). The linear map $\mathcal S_{k,N,\chi,\d}$ is the so-called {\em $\d$-th Shimura lifting}. 

We denote by $\theta_{k,N,\chi,\d}\colon S_{2k}(N,\chi^2) \to S_{k+1/2}^+(N,\chi)$ the adjoint map with respect to the Petersson product, meaning that for all $g \in S_{k+1/2}^+(N,\chi)$ and $f \in S_{2k}(N,\chi^2)$
\[
    \langle g, \theta_{k,N,\chi,\d}(f)\rangle = \langle \mathcal S_{k,N,\chi,\d}(g), f\rangle.
\]
This implies that for any $f \in S_{2k}(N,\chi^2)$ one has
\begin{align*}
    \theta_{k,N,\chi,\d}(f) & =
    \langle f(z), \overline{\Omega_{k,N,\chi}(-\bar z,\tau; \d)} \rangle 
    =
    i_N^{-1} \int_{\Gamma_0(N)\backslash\mathfrak H} f(z)\Omega_{k,N,\chi}(-\bar z,\tau; \d) y^{2k-2}dxdy = \\
    & = i_Nc_{k,\d,\chi}^{-1} \hspace{-15pt} \sum_{\substack{m\geq 1,\\\epsilon(-1)^km\equiv 0,1 (4)}} \hspace{-15pt} m^{k-1/2} \left(\sum_{t\mid N_1} \mu(t)\chi_{\d} \bar{\chi}_0(t) t^{k-1} \langle f, f_{N_0,N_1/t}^{k,\bar{\chi}_0}(-tz; \d, \epsilon(-1)^k m)\rangle \right) q^m,
\end{align*}
where the last equality uses the above mentioned expression of the kernel function in terms of Poincar\'e series. In particular, an explicit expression for $\theta_{k,N,\chi,\d}(f)$ can be determined by computing the Petersson products 
\[
\langle f, f_{N_0,N_1/t}^{k,\bar{\chi}_0}(-tz; \d, \epsilon(-1)^k m)\rangle,
\]
for $m\geq 1$ with $\epsilon(-1)^k m \equiv 0,1 \pmod 4$. Firstly, using property ii) listed above for the functions $f_{k,N_0,u}$, we see that 
\[
f_{N_0,N_1/t}^{k,\bar{\chi}_0}(-tz; \d, \epsilon(-1)^k m) = f_{N_0,N_1/t}^{k,\bar{\chi}_0}(tz; \d, \epsilon(-1)^k m).
\]
Secondly, using the bijection in \eqref{bijection-t} and that $\omega_{\d}(Q_t) = \chi_{\d}(t^2)\omega_{\d}(Q)$ by \eqref{genuschar-relation-t}, we deduce that
\begin{equation}\label{fkN-relation-t}
f_{N_0,N_1t}^{k,\bar{\chi}_0}(z; \d, \epsilon(-1)^k mt^2) = \chi_{\d}(t^2)f_{N_0,N_1/t}^{k,\bar{\chi}_0}(-tz; \d, \epsilon(-1)^k m)
\end{equation}
for all divisors $t$ of $N_1$. Therefore, we may rewrite the above expression for $\theta_{k,N,\chi,\d}(f)$ as
\[
    \theta_{k,N,\chi,\d}(f) = i_N c_{k,\d,\chi}^{-1}  \hspace{-15pt} \sum_{\substack{m\geq 1,\\\epsilon(-1)^km\equiv 0,1 (4)}} \hspace{-15pt} m^{k-1/2} \left(\sum_{t\mid N_1} \mu(t)\chi_{\d} \bar{\chi}_0(t) t^{k-1} \langle f, f_{N_0,N_1t}^{k,\bar{\chi}_0}(z; \d, \epsilon(-1)^k mt^2)\rangle \right) q^m.
\]
Finally, proceeding similarly as in \cite[p. 265-266]{Kohnen-FourierCoefficients}, one can check that for $t \mid N_1$
\begin{equation}\label{fkN-rkN}
\langle f, f_{N_0,N_1t}^{k,\bar{\chi}_0}(z; \d, \epsilon(-1)^k mt^2)\rangle = i_{Nt}^{-1} \pi \binom{2k-2}{k-1} 2^{-2k+2}(|\d|mt^2)^{-k+1/2}r_{k,Nt,\chi}(f; \d, \epsilon(-1)^kmt^2),
\end{equation}
where for any discriminant $\dd$ with $\d\dd>0$ we set
\begin{align*}
r_{k,Nt,\chi}(f; \d, \dd) & := \sum_{Q \in \mathcal L_{Nt}(N_0^2\d\dd)/\Gamma_0(Nt)} \omega_{\d}(Q) \chi_0(Q)\int_{C_Q} f(z) Q(z,1)^{k-1} dz.
\end{align*}
Here, for each $Q = [a,b,c]$, $C_Q$ denotes the image in $\Gamma_0(N) \backslash \mathfrak H$ of a geodesic in the upper half plane associated with $Q$. Namely, consider the semicircle $S_Q$ in the complex upper half plane defined by the equation $a|z|^2+b \mathrm{Re}(z)+c = 0$, and denote by $\omega_Q$, $\omega_Q' \in \P^1(\R)$ the pair of points 
\[
(\omega_Q, \omega_Q') := \begin{cases}
\left(\frac{b-\sqrt{\mathrm{disc}(Q)}}{2a}, \frac{b+\sqrt{\mathrm{disc}(Q)}}{2a} \right) & \text{if } a \neq 0, \\
(-c/b, i\infty) & \text{if } a = 0, b > 0, \\
(i\infty, -c/b) & \text{if } a = 0, b < 0.
\end{cases}
\]
Notice that $\omega_Q$ and $\omega_Q'$ are the endpoints of the semicircle. When $\mathrm{disc}(Q)$ is a perfect square, we let $C_Q$ be the image of $S_Q$ (oriented from $\omega_Q$ to $\omega_Q'$) in $\Gamma_0(N)\backslash\mathfrak H$. If $\mathrm{disc}(Q)$ is not a perfect square, then let $\Gamma_0(N)_Q$ be the stabilizer of $Q$ in $\Gamma_0(N)$, $\Gamma_0(N)_Q^+$ be its index two subgroup of positive trace elements, and 
\[
\gamma_Q = \begin{pmatrix} r & s \\ t & u\end{pmatrix} \in \Gamma_0(N)_Q^+
\]
be the unique generator such that $r-t\omega_Q > 1$. Then, we let $C_Q$ be the image in $\Gamma_0(N)\backslash\mathfrak H$ of the oriented geodesic path from $\gamma_Q(i\infty)$ to $i\infty$. Note that with this construction, the endpoints of $C_Q$ are always rational cusps. We will write 
\[
I_{k,\chi}(f,Q) := \chi_0(Q) \int_{C_Q} f(z) Q(z,1)^{k-1} dz,
\]
so that \begin{equation}\label{rkN-def}
r_{k,Nt,\chi}(f;\d,\dd) = \sum_{Q \in \mathcal L_{Nt}(N_0^2\d\dd)/\Gamma_0(Nt)} \omega_{\d}(Q)I_{k,\chi}(f,Q).
\end{equation}

With this, one eventually concludes that \begin{equation}\label{dShintani-withcharacter}
\theta_{k,N,\chi,\d}(f) = C(k,\chi,\d) \sum_{\substack{m\geq 1,\\\epsilon(-1)^km\equiv 0,1(4)}} \left(\sum_{t\mid N_1} \mu(t)\chi_{\d}\bar{\chi}_0(t) t^{-k-1} r_{k,Nt,\chi}(f;\d,\epsilon(-1)^k m t^2)\right) q^m,
\end{equation}
where 
\begin{equation}\label{constant:Ckchid}
    C(k,\chi,\d) := c_{k,\d,\chi}^{-1} \pi \binom{2k-2}{k-1}2^{-2k+2}|\d|^{1/2-k} = (-1)^{[k/2]} \epsilon^{k+1/2} 2^k N_0^{k-1} \frac{\mathfrak g(\chi_{\d}\bar{\chi}_0)}{\mathfrak g(\chi_{\d})}.
\end{equation}
When the character $\chi$ is trivial, we will denote this constant by $C(k,\d)=(-1)^{[k/2]}2^k$. 

\begin{remark} \label{rmk:vanishing}
    One can easily check that the quantities $r_{k,Nt,\chi}(f; \d,\dd)$ in \eqref{rkN-def} are all zero when $\epsilon(-1)^k\d < 0$, so that $\theta_{k,N,\chi,\d}(f)$ vanishes identically. Therefore we may assume that $\d$ is chosen such that $\epsilon(-1)^k\d > 0$.
\end{remark}

\begin{remark}
The expression for $\theta_{k,N,\chi,\d}$ in \cite[Theorem 3.2]{KojimaTokuno} reads slightly different, with $t^{-k}$ instead of $t^{-k-1}$, since they use a slight variation of the sum of cycle integrals $r_{k,Nt,\chi}$, considering equivalence by $\Gamma_0(N)$ instead of $\Gamma_0(Nt)$. The two sums yield the same result, and the reason for the extra factor $t^{-1}$ showing up here is that $i_{Nt} = t\cdot i_N$. For trivial character, \eqref{dShintani-withcharacter} recovers the expression in \cite[Eq. (8)]{Kohnen-FourierCoefficients}, where the constant $C(k,\mathbf 1,\d) = (-1)^{[k/2]} 2^k$ seems to be missing.
\end{remark}

When $f \in S_{2k}(N,\chi^2)$ is {\em new}, the expression in \eqref{dShintani-withcharacter} gets simplified. Indeed, equation \eqref{fkN-relation-t} shows that the forms $f_{N_0,N_1t}^{k,\bar{\chi}_0}(z;\d,\epsilon(-1)^kmt^2)$ are {\em old} when $t>1$, and hence the left hand side of \eqref{fkN-rkN} vanishes and $r_{k,Nt,\chi}(f;\d,\epsilon(-1)^kmt^2) = 0$ for $t > 1$. Therefore, if $f \in S_{2k}(N,\chi^2)$ is {\em new} one finds 
\begin{equation}\label{dShintani-withcharacter-fnew}
\theta_{k,N,\chi,\d}(f) = C(k,\chi,\d) \sum_{\substack{m\geq 1,\\\epsilon(-1)^km\equiv 0,1(4)}} r_{k,N,\chi}(f;\d,\epsilon(-1)^k m) q^m.
\end{equation}
In particular, if $m\geq 1$ is such that $\epsilon (-1)^k m \equiv 0, 1$ (mod $4$), then the $m$-th coefficient of $\theta_{k,N,\chi,\d}(f)$ is just 
\[
a_m(\theta_{k,N,\chi,\d}(f)) = C(k,\chi,\d) r_{k,N,\chi}(f;\d,\epsilon(-1)^k m).
\]

\subsection{Fourier coefficients and $L$-values}

It is well-known that Fourier coefficients of half-integral weight modular forms encode special values of (twisted) $L$-series of integral weight modular forms. We will review this phenomenon in this paragraph, assuming for simplicity of exposition that $N$ and $\d$ are relatively prime.

Indeed, suppose first that $N$ is odd and squarefree, and let $f \in \sum a_n(f) q^n \in S_{2k}(N)$ be a normalized Hecke eigenform of weight $2k$, level $\Gamma_0(N)$, and trivial nebentype character. For each prime divisor $\ell$ of $N$, let $W_{\ell}$ denote the $\ell$-th Atkin--Lehner involution, and $w_{\ell} \in \{\pm 1\}$ be the Atkin--Lehner eigenvalue of $f$ at $\ell$, so that $f | W_{\ell} = w_{\ell} f$. Let $\d$ be a fundamental discriminant with $(-1)^k\d > 0$ and such that $\mathrm{gcd}(N,\d)=1$. Let $L(f,\chi_{\d},s)$ be the complex $L$-series of $f$ twisted by the quadratic character $\chi_{\d}$. This $L$-series has holomorphic continuation to the whole complex plane, yielding a completed $L$-series $\Lambda(f,\chi_{\d},s)$ that satisfies the functional equation 
\[
\Lambda(f,\chi_{\d},s) = (-1)^k \chi_{\d}(-N)w_N \Lambda(f,\chi_{\d},2k-s),
\]
where $w_N \in \{\pm 1\}$ is the product of all the $w_{\ell}$ for $\ell \mid N$ prime. In this setting, Kohnen's formula in \eqref{KohnenFormula-intro} relates the $|\d|$-th Fourier coefficient of any non-zero half-integral weight modular form in Shimura--Shintani correspondence with $f$ to the special value $L(f,\chi_{\d},k)$, provided that $\chi_{\d}(\ell)=w_{\ell}$ for all primes $\ell$ dividing $N$. A key point in the proof of that formula is the identity
\begin{equation}\label{rkN-Lvalue}
r_{k,N}(f;\d,\d) = 2^{\nu(N)}|\d|^k(k-1)! \cdot \frac{L(f,\chi_{\d},k)}{(2\pi i)^k\mathfrak g(\chi_{\d})},
\end{equation}
which the reader can check in p. 243 of op. cit. (note that $\mathfrak g(\chi_{\d}) = (-1)^k|\d|/\mathfrak g(\chi_{\d})$). Assuming that $f$ is new, we have $a_{|\d|}(\theta_{k,N,\d}(f)) = C(k,\d)r_{k,N}(f;\d,\d)$, and hence we deduce that
\begin{equation}\label{ad-Lvalue-basic}
    a_{|\d|}(\theta_{k,N,\d}(f))) = C(k,\d)2^{\nu(N)}|\d|^k(k-1)! \cdot \frac{L(f,\chi_{\d},k)}{(2\pi i)^k\mathfrak g(\chi_{\d})}.
\end{equation}

Kohnen's formula in \eqref{KohnenFormula-intro} has been generalized by Kojima--Tokuno to the case where $f$ has non-squarefree level and non-trivial nebentype character, under a suitable multiplicity one assumption. We refer the reader to \cite[Theorems 4.1, 4.2]{KojimaTokuno} for the details. We will rather focus on the identity analogous to \eqref{rkN-Lvalue}, which is also a key step in the proof of the generalization of \eqref{KohnenFormula-intro} but it holds unconditionally. In order to describe such identity, we need to introduce some notation.

Let $f \in S_{2k}^{\mathrm{new}}(N,\chi^2)$ be a normalized newform of weight $2k$, odd level $N\geq 1$, and nebentype character $\chi^2$. As before, $N_0$ will denote the conductor of $\chi$, $\chi_0$ the primitive character associated with $\chi$, and $\epsilon = \chi(-1)$. We continue to assume that $\mathrm{gcd}(N_0,N_1) = 1$, where $N_1 = N/N_0$, and hypothesis \eqref{hypothesisk=1}. With this, let $\d$ be a fundamental discriminant such that $\epsilon(-1)^k\d > 0$, and assume further that\footnote{Notice that this is stronger than our previous assumption $\mathrm{gcd}(N_0,\d)=1$.} $\mathrm{gcd}(N,\d)=1$. We will briefly explain how does one relate $r_{k,N,\chi}(f;\d,\d)$, as defined in \eqref{rkN-def}, to the special value $L(f,\chi_{\d}\bar{\chi}_0,k)$ of the $L$-series of $f$ twisted by $\chi_{\d}\bar{\chi}_0$, yielding the analogous identity to \eqref{rkN-Lvalue} above. 

For each positive divisor $d$ of $N_1$, with $\mathrm{gcd}(d,N_1/d) = 1$ (we write $d || N_1$), let $W_d$ be the $d$-th Atkin--Lehner element in $\GL_2^+(\R)$ defined by any matrix 
\[
    W_d = \frac{1}{\sqrt{d}}\begin{pmatrix} d & \alpha_d \\  N & \beta_d d \end{pmatrix} \quad \text{where } \alpha_d, \beta_d \in \Z \text{ are such that }  \beta_d d^2 - \alpha_d N = d.
\]
Since $d$ divides exactly $N_1$, observe from the definition that $\beta_d \equiv d^{-1}$ modulo $N/d$. 

Since we are considering divisors of $N_1$, the above elements $W_d$ act as automorphisms of $S_{2k}(N,\chi^2)$ via the weight $2k$ slash operator. Furthermore, since $f$ is a normalized newform, for each $d$ as above there exists a normalized newform $f_d \in S_{2k}(N,\chi^2)$ and a non-zero constant $w_d(f)$ such that 
\[
    f |_{2k} W_d = w_d(f) f_d.
\]
These constants are multiplicative, meaning that if $dd' || N_1$ with $\mathrm{gcd}(d,d')=1$, then $w_{dd'}(f) = w_d(f)w_{d'}(f)$.

\begin{proposition}[Kojima--Tokuno]\label{prop:rkN-Lvalue-KT}
Let $k\geq 1$ be an integer and $N\geq 1$ be an odd integer. Let $\chi$ be a Dirichlet character modulo $N$, with conductor $N_0$ and associated primitive character $\chi_0$, and let $\epsilon = \chi(-1)$. Assume $\mathrm{gcd}(N_0,N_1) = 1$, where $N_1 = N/N_0$, and let $\d$ be a fundamental discriminant such that $\mathrm{gcd}(N,\d)=1$ and $\epsilon(-1)^k\d > 0$. If $f \in S_{2k}(N,\chi^2)$ is a normalized Hecke eigenform, then
\begin{equation}\label{rkN-Lvalue-KT}
    r_{k,N,\chi}(f;\d,\d) =  \chi_{\d}(-1) R_{\d}(f) (-1)^k |\d|^k N_0^k (k-1)! \cdot \frac{L(f,\chi_{\d}\bar{\chi}_0,k)}{(2\pi i)^k\mathfrak g(\chi_{\d}\bar{\chi}_0)},
\end{equation}
where 
\[
    R_{\d}(f) = \prod_{\ell^e || N_1}\left(1+\chi_{\d}\bar{\chi}_0(\ell^e)w_{\ell^e}(f) \frac{1-\chi_{\d}\bar{\chi}_0(\ell)a_{\ell}(f)\ell^{-k}}{1-\chi_{\d}\chi_0(\ell)\overline{a_{\ell}(f)}\ell^{-k}}\right).
\]
Here, the product is over the prime divisors $\ell$ of $N_1$, and for each such prime, with $\ell^e || N_1$, $w_{\ell^e}(f)$ denotes the constant associated with $W_{\ell^e}$ as above.
\end{proposition}
\begin{proof}
The elements $W_d$ also act on quadratic forms, by the rule 
\[
Q\circ W_d := {}^t W_d Q W_d, \quad Q \in \mathcal L_N(N_0^2\d^2).
\]
It is straightforward to check that $Q\circ W_d$ belongs again to $\mathcal L_N(N_0^2\d^2)$. Since $W_d$ normalizes $\Gamma_0(N)$, the map $Q \mapsto Q\circ W_d$ establishes a bijection from $\mathcal L_N(N_0^2\d^2)/\Gamma_0(N)$ to itself. In addition, one can take the set 
\[
\left \lbrace Q_{\mu} \circ W_d: \mu \in \Z/\d\Z \times (\Z/ N_0\Z)^{\times}, \, d || N_1 \right \rbrace, \quad Q_{\mu} = \begin{pmatrix} 0 & \d N_0/2 \\ \d N_0/2 & \mu \end{pmatrix},
\]
as a complete set of representatives for $\mathcal L_N(N_0^2\d^2)/\Gamma_0(N)$. One can then rewrite the sum defining $r_{k,N,\chi}(f;\d,\d)$ as a sum over the above set, and use the explicit representatives to compute the involved cycle integrals to eventually deduce the claimed formula (cf. \cite[(4-21), (4-22)]{KojimaTokuno}).
\end{proof}

Note that the identity in $\eqref{rkN-Lvalue-KT}$ reduces to \eqref{rkN-Lvalue} when one assumes that $\chi$ is trivial, $N$ is squarefree and $\chi_{\d}(\ell) = w_{\ell}(f)$ for all primes $\ell \mid N$. Indeed, when $\chi$ is trivial we have $N_0=1$ and $\chi_{\d}(-1)(-1)^k = 1$, and the Fourier coefficients of $f$ are real, thus the assumption that $N$ is squarefree yields $R_{\mathfrak d}(f) =2^{\nu(N)}$. Finally, recalling that when $f$ is new we have $a_{|\d|}(\theta_{k,N,\chi,\d}(f)) = C(k,\chi,\d) r_{k,N,\chi}(f;\d,\d)$, with $C(k,\chi,\d)$ as in \eqref{constant:Ckchid}, we immediately deduce the following:

\begin{corollary}\label{cor:ad-Lvalue}
With the same assumptions as in the theorem, if $f$ is new then
\begin{equation}\label{ad-Lvalue}
    a_{|\d|}(\theta_{k,N,\chi,\d}(f)) =  C(k,\chi,\d) \chi_{\d}(-1)R_{\d}(f)(-1)^k|\d|^kN_0^k(k-1)!\cdot \frac{L(f,\chi_{\d}\bar{\chi}_0,k)}{(2\pi i)^k\mathfrak g(\chi_{\d}\bar{\chi}_0)}.
\end{equation}
\end{corollary}

Again, note that \eqref{ad-Lvalue} reduces to \eqref{ad-Lvalue-basic} when $\chi$ is trivial, $N$ is squarefree, and the same assumption on Atkin--Lehner eigenvalues as above.

\section{\texorpdfstring{$\d$}{d}-th Shintani liftings of \texorpdfstring{$p$}{p}-stabilized newforms}  \label{Sec:p-stabilizations}

We investigate in this section the $\d$-th Shintani liftings of forms belonging to the old subspace in $S_{2k}(Np,\chi^2)$ arising from newforms of level $N$. More precisely, given a newform $f \in S_{2k}^{new}(N,\chi^2)$, we study the $\d$-th Shintani liftings of forms in the $f$-isotypic component $S_{2k}(Np,\chi^2)[f]$ spanned by $f$ and $V_pf$, where $V_p$ denotes the usual operator $V_pf(z) = f(pz)$. In particular, this will allow us to relate the $\d$-th Shintani lifting (in level $Np$) of the (ordinary) $p$-stabilization of a newform $f \in S_{2k}^{new}(N,\chi^2)$. To achieve this goal, we first need to discuss a detailed study of the classification of integral binary quadratic forms up to equivalence by congruence subgroups.

\subsection{Discussion on integral binary quadratic forms}\label{sec:discussion-quadforms}

Let $M\geq 1$ be an odd integer, and let $\Delta$ be a discriminant. Recall that we have introduced the set
\[
 \mathcal Q_M(\Delta) = \{[a,b,c]\in \mathcal Q(\Delta): \, a \equiv 0 \text{ (mod }M\text{)}\},
\]
on which $\Gamma_0(M)$ acts. Our aim is to classify the $\Gamma_0(M)$-orbits in this set, following the discussion in \cite{GKZ} and emphasizing some aspects that will be of special interest for us. As a matter of notation, we will write forms in $\mathcal Q_M(\Delta)$ as $[aM,b,c]$, where $a, b, c$ are integers.  Yet another invariant for the action of $\Gamma_0(M)$ on $\mathcal Q_M(\Delta)$ is the residue class of $b$ modulo $2M$. Notice that not every residue class is allowed: one must have $b^2 \equiv \Delta \pmod{4M}$. Setting 
\[
 R_M(\Delta) := \{\varrho \in \Z/2M\Z: \, \varrho^2 \equiv \Delta \text{ (mod }4M\text{)}\},
\]
and defining 
\[
  \mathcal Q_{M,\varrho}(\Delta) = \{[a,b,c]\in \mathcal Q(\Delta): \, a \equiv 0  \text{ (mod }M\text{)}, \, b\equiv \varrho \text{ (mod }2M\text{)}\}
\]
for each $\varrho \in R_M(\Delta)$, observe that one has a $\Gamma_0(M)$-invariant decomposition
\[
 \mathcal Q_M(\Delta) = \bigsqcup_{\varrho \in R_M(\Delta)} \mathcal Q_{M,\varrho}(\Delta).
\]
The sets $\mathcal Q_{M,\varrho}(\Delta)$ being $\Gamma_0(M)$-invariant, we are reduced to study their $\Gamma_0(M)$-orbits. We further define the subset of {\em $\Gamma_0(M)$-primitive} forms in $\mathcal Q_{M,\varrho}(\Delta)$ by 
\[
  \mathcal Q_{M,\varrho}^1(\Delta) = \{[aM,b,c]\in \mathcal Q_{M,\varrho}(\Delta): \, \mathrm{gcd}(a,b,c)=1\}.
\]
Note that with this notation, and the analogous notation for $\mathcal Q_{M,\rho}^d(\Delta)$, we have
\[
 \mathcal Q_{M,\rho}^d(\Delta) = d \mathcal Q_{M,\rho}^1(\Delta/d^2).
\]

\begin{remark}
Observe in the above definition that we consider the greatest common divisor of $a$, $b$, and $c$, not of $aM$, $b$, and $c$. 
\end{remark}

One has a $\Gamma_0(M)$-invariant bijection of sets 
\[
 \mathcal Q_{M,\varrho}(\Delta) = \bigsqcup_{d^2\mid \Delta} \bigsqcup_{\substack{\ell \in R_M(\Delta/d^2),\\ d\ell \equiv \varrho \, (2M)}} d\cdot \mathcal Q_{M,\ell}^1(\Delta/d^2),
\]
where $d$ varies over the positive integers such that $d^2$ divides $\Delta$, and $\ell$ varies over the elements in $R_M(\Delta/d^2)$ such that $d\ell \equiv \varrho$ modulo $2M$. Via this last decomposition, it is enough to study the $\Gamma_0(M)$-orbits in the sets of the form $\mathcal Q_{M,\varrho}^1(\Delta)$, where $M\geq 1$ is an integer, $\Delta$ is a discriminant, and $\varrho \in R_M(\Delta)$.

Continue to fix parameters $M$ and $\Delta$ as before. For each $\varrho \in R_M(\Delta)$, associated with the set $\mathcal Q_{M,\varrho}^1(\Delta)$ there is the integer 
\[
m_{\varrho} := \mathrm{gcd}\left(M,\varrho,\frac{\varrho^2-\Delta}{4M}\right),
\]
which is well-defined even if $\varrho$ is only defined modulo $2M$. Indeed, replacing $\varrho$ by $\varrho + 2M$ replaces $\frac{\varrho^2-\Delta}{4M}$ by $\frac{\varrho^2-\Delta}{4M}+\varrho+M$. Thus $m_{\varrho}$ is an invariant of the subset $\mathcal Q_{M,\varrho}^1(\Delta)$. One can check that 
\[
m_{\varrho} = \mathrm{gcd}(M,b,ac) \qquad \text{for any } \, Q = [aM,b,c] \in \mathcal Q_{M,\varrho}^1(\Delta).
\]
In particular, notice that $m_{\varrho} \mid \mathrm{gcd}(M,\Delta)$. Using this, one can further decompose $m_{\varrho}$ as
\[
m_{\varrho} = m_1m_2, \quad \text{ where } \quad m_1 = \mathrm{gcd}(M,b,a), \quad m_2 =\mathrm{gcd}(M,b,c).
\]
Notice that $\mathrm{gcd}(m_1,m_2)=1$ because $\mathrm{gcd}(a,b,c)=1$. In view of this, we have a $\Gamma_0(M)$-invariant decomposition 
\[
\mathcal Q^1_{M,\varrho}(\Delta) = \bigsqcup_{m_1, m_2} \mathcal Q_{M,\varrho,m_1,m_2}^1(\Delta),
\]
where $m_1, m_2$ run over the pairs of positive divisors of $m_{\varrho}$ which satisfy $m_1m_2 = m_{\varrho}$ and $\mathrm{gcd}(m_1,m_2) = 1$. There are $2^{\nu}$ such pairs, where $\nu$ is the number of prime factors of $m_{\varrho}$.

The following statement is the Proposition in p. 505 of \cite{GKZ}:

\begin{proposition}\label{prop:formsGKZ}
 Let $m_{\varrho}$ be defined as above and fix any decomposition $m_{\varrho} = m_1m_2$ with $m_1,m_2>0$ integers such that $\mathrm{gcd}(m_1,m_2)=1$. Fix also any decomposition $M = M_1M_2$ of $M$ into coprime factors satisfying $\mathrm{gcd}(m_1,M_2)=\mathrm{gcd}(m_2,M_1)=1$. Then the map $Q = [aM,b,c] \mapsto \tilde Q = [aM_1,b,cM_2]$ yields a one-to-one correspondence 
 \[
  \mathcal Q_{M,\varrho,m_1,m_2}^1(\Delta)/\Gamma_0(M) \, \stackrel{1:1}{\longrightarrow} \, \mathcal Q^1(\Delta)/\Gamma_0(1).
 \]
 In particular, $|\mathcal Q_{M,\varrho}^1(\Delta)/\Gamma_0(M)| = 2^{\nu}|\mathcal Q^1(\Delta)/\Gamma_0(1)|$ where $\nu$ is the number of prime factors of $m_{\varrho}$.
\end{proposition}

Notice that this proposition finishes the precise description of  the set of integral binary quadratic forms $\mathcal Q_M(\Delta)$, up to $\Gamma_0(M)$-action. Indeed, as a summary of the above discussion we have a disjoint union of $\Gamma_0(M)$-invariant sets
\begin{equation}\label{decompositionQs}
\mathcal Q_M(\Delta) = \bigsqcup_{d^2\mid\Delta} \bigsqcup_{\varrho \in R_M(\Delta/d^2)} \bigsqcup_{(m_1,m_2)} d \cdot \mathcal Q_{M,\varrho,m_1,m_2}^1(\Delta/d^2),
\end{equation}
where $d$ varies over the positive integers such that $d^2\mid \Delta$, and for each $\varrho\in R_M(\Delta/d^2)$ the pair $(m_1,m_2)$ ranges over the pairs of coprime positive integers with $m_1m_2=m_{\varrho}$. For each tuple $(d,\varrho,m_1,m_2)$, the set of classes $\mathcal Q_{M,\varrho,m_1,m_2}^1(\Delta/d^2)/\Gamma_0(M)$ is in bijection with $\mathcal Q^1(\Delta/d^2)/\Gamma_0(1)$, which is a class group {\em à la Gauss}.

Fix an odd integer $N\geq 1$, an odd prime $p$ not dividing $N$, and a discriminant $\Delta$ divisible by $p$. To close this paragraph, we will use Proposition \ref{prop:formsGKZ} to compare certain sets of orbits of quadratic forms of a given discriminant by congruence groups of different levels $N$ and $Np$. To do so, we first need to compare the sets $R_N(\Delta)$ and $R_{Np}(\Delta)$. By our hypotheses, $\Delta$ is a square modulo $Np$ if and only if it is a square modulo $N$, so that $R_{Np}(\Delta)$ is non-empty if and only if $R_{N}(\Delta)$ is non-empty. We have the following:

\begin{lemma} \label{lemma:mrho}
If $R_N(\Delta)$ is non-empty, then the natural morphism $\Z/2Np\Z \to \Z/2N\Z$ induces a bijection
 \[
  r_d\colon R_{Np}(\Delta) \, \longrightarrow \, R_N(\Delta),
 \]
and the following assertions hold:
 \begin{item}
  \item[i)] if $p^2 \nmid \Delta$, then $m_{\varrho'} = m_{r_d(\varrho')}$ for all $\varrho' \in R_{Np}(\Delta)$;
  \item[ii)] if $p^2 \mid \Delta$, then $m_{\varrho'} = p m_{r_d(\varrho')}$  for all $\varrho' \in R_{Np}(\Delta)$.
 \end{item}
\end{lemma}

\begin{proof}
	From the Chinese remainder theorem it is clear that the set of elements in $\Z/2Np\Z$ divisible by $p$ is in bijection with the elements of $\Z/2N\Z$ and by hypothesis  $R_{Np}(\Delta)$ contains only elements divisible by $p$. The first part of the statement follows.
	
	For the second part, write $\varrho = r_d(\varrho')$, and let $R$ be an integer representing both $\varrho \in R_{N}(\Delta) \subset \Z/2N\Z$ and $\varrho' \in R_{Np}(\Delta) \subset \Z/2Np\Z$.  Since $R\equiv \varrho' \pmod{2Np}$, and we assume $p \mid \Delta$, we have that $p\mid R$. The invariants $m_{\varrho}$ and $m_{\varrho'}$ associated with the sets 
 \[
  \mathcal Q_{N,\varrho}^1(\Delta) \quad \text{and} \quad \mathcal Q_{Np,\varrho'}^1(\Delta)
 \]
 are defined as 
 \[
  m_{\varrho} = \mathrm{gcd}\left(N,R,\frac{R^2 - \Delta}{4N}\right) \quad \text{and} \quad 
  m_{\varrho'} = \mathrm{gcd}\left(Np,R,\frac{R^2 - \Delta}{4Np}\right),
 \]
 respectively. Since $p \nmid N$, we have
 \[
 	m_{\varrho'} = \mathrm{gcd}\left(p,R,\frac{R^2 - \Delta}{4Np}\right) \mathrm{gcd}\left(N,R,\frac{R^2 - \Delta}{4Np}\right).
 \]
 When $p^2\nmid \Delta$, the first term is trivial, hence $m_{\varrho'} = m_{\varrho}$. If $p^2\mid \Delta$ instead, then $p\mid (R^2-\Delta)/4Np$ and the first term is $p$, so that  $m_{\varrho'} = p m_{\varrho}$.

\end{proof}

We will be interested in comparing sets of the form 
\[
\mathcal L_{N}(N_0^2\Delta)/\Gamma_0(N) \quad \text{and} \quad \mathcal L_{Np}(N_0^2\Delta)/\Gamma_0(Np),
\]
from Section \ref{Sec:Shintani-Lifting}, where the discriminant $\Delta$, the prime $p$ and the level $N=N_0N_1$ satisfy the following conditions:
\begin{itemize}
 \item[i)] $N_0$ and $N_1$ are coprime integers, i.e. $\mathrm{gcd}(N_0,N_1) = 1$;
 \item[ii)] $p$ is an odd prime such that $p\nmid N$ and $p\mid \Delta$.
\end{itemize}
Recall also that the sets `$\mathcal L$' consist of forms in the sets `$\mathcal Q$' with  $\mathrm{gcd}(N_0,c)=1$. We will naturally write $\mathcal L_{N,\varrho}^1(\cdot)$, $\mathcal L_{N,\varrho,m_1,m_2}^1(\cdot)$, etc., for the intersection of $\mathcal Q_{N,\varrho}^1(\cdot)$, $\mathcal Q_{N,\varrho,m_1,m_2}^1(\cdot)$, etc. with $\mathcal L_N(\cdot)$. An important observation is that, via the decomposition in \eqref{decompositionQs}, the subset $\mathcal L_N(N_0^2\Delta)$ (resp. $\mathcal L_{Np}(N_0^2\Delta)$) of $\mathcal Q_N(N_0^2\Delta)$ (resp. $\mathcal Q_{Np}(N_0^2\Delta)$) is precisely the union of those subsets $d\cdot \mathcal Q_{N,\varrho,m_1,m_2}^1(N_0^2\Delta/d^2)$ (resp. $d\cdot \mathcal Q_{Np,\varrho,m_1,m_2}^1(N_0^2\Delta/d^2)$) with $\mathrm{gcd}(m_2,N_0)=1$ and $\mathrm{gcd}(d,N_0)=1$. In particular, \eqref{decompositionQs} yields the analogous $\Gamma_0(N)$-invariant (resp. $\Gamma_0(Np)$-invariant) decompositions
\begin{equation}\label{decompositionLs:N}
 \mathcal L_N(N_0^2\Delta) = \bigsqcup_{\substack{d^2\mid \Delta,\\\mathrm{gcd}(d,N_0)=1}} \bigsqcup_{\varrho \in R_N(N_0^2\Delta/d^2)} \bigsqcup_{\substack{(m_1,m_2),\\ \mathrm{gcd}(m_2,N_0)=1}} d \cdot \mathcal Q_{N,\varrho,m_1,m_2}^1(N_0^2\Delta/d^2)
\end{equation}
and 
\begin{equation}\label{decompositionLs:Np}
 \mathcal L_{Np}(N_0^2\Delta) = \bigsqcup_{\substack{d^2\mid \Delta,\\\mathrm{gcd}(d,N_0)=1}} \bigsqcup_{\varrho' \in R_{Np}(N_0^2\Delta/d^2)} \bigsqcup_{\substack{(m_1',m_2'),\\ \mathrm{gcd}(m'_2,N_0)=1}} d \cdot \mathcal Q_{Np,\varrho',m_1',m_2'}^1(N_0^2\Delta/d^2).
\end{equation}

Here, $d$ varies over the positive integers such that $d^2 \mid \Delta$ and $\mathrm{gcd}(d,N_0)=1$, and for each $\varrho \in R_N(N_0^2\Delta/d^2)$ (resp. $\varrho' \in R_{Np}(N_0^2\Delta/d^2)$) the pair $(m_1,m_2)$ (resp. $(m_1',m_2')$) ranges over the pairs of coprime positive integers with $m_1m_2=m_{\varrho}$ (resp. $m_1'm_2'=m_{\varrho'}$) satisfying the extra condition $\mathrm{gcd}(N_0,m_2)=1$ (resp. $\mathrm{gcd}(N_0,m_2') = 1$). 

For each choice of $d$, $\varrho$, and $(m_1,m_2)$ arising in the decomposition \eqref{decompositionLs:N}, by virtue of Proposition \ref{prop:formsGKZ} we may fix a bijection
\begin{align*}
 \beta_{m_1,m_2}^{d,\varrho}\colon \mathcal Q_{N,\varrho,m_1,m_2}^1(N_0^2\Delta/d^2)/\Gamma_0(N) \, & \stackrel{1:1}{\longrightarrow} \, \mathcal Q^1(N_0^2\Delta)/\Gamma_0(1)\\
 [aN,b,c] \, & \longmapsto [aM_1,b,M_2],
\end{align*}
where $N=M_1M_2$ is any decomposition of $N$ as a product of two coprime integers such that $\mathrm{gcd}(m_1,M_2)=\mathrm{gcd}(m_2,M_1)=1$.

By virtue of lemma \ref{lemma:mrho}, when $p^2\nmid \Delta$ the sets of pairs $(m_1,m_2)$ and $(m_1',m_2')$ arising in the decompositions \eqref{decompositionLs:N} and \eqref{decompositionLs:Np} are exactly the same, and Proposition \ref{prop:formsGKZ} easily implies the following:

\begin{corollary}\label{cor:nop2}
 Let $N$, $N_0$, $\Delta$, and $p$ be as above, with $p^2 \nmid \Delta$. Then both $\mathrm{id}\colon [aNp,b,c] \mapsto [(ap)N,b,c]$ and $\tau\colon [aNp,b,c] \mapsto [aN,b,cp]$ induce bijections
 \[
  \mathcal L_{Np}(N_0^2\Delta)/\Gamma_0(Np) \, \stackrel{1:1}{\longrightarrow} \, \mathcal L_{N}(N_0^2\Delta)/\Gamma_0(N).
 \]
\end{corollary}
\begin{proof}
 Fix a positive integer $d$ with $d^2 \mid \Delta$ and $\mathrm{gcd}(d,N_0)=1$, and any element $\varrho \in R_N(N_0^2\Delta/d^2)$. Let $\varrho' \in R_{Np}(N_0^2\Delta/d^2)$ be the unique element such that $r_d(\varrho')=\varrho$. By lemma \ref{lemma:mrho}, $m_{\varrho'} = m_{\varrho}$, and if $m_1$, $m_2$ are coprime integers with $m_1m_2 = m_{\varrho}$ and $\mathrm{gcd}(m_2,N_0)=1$, Proposition \ref{prop:formsGKZ} tells us that both arrows in 
 \[
  \mathcal Q_{N,\varrho,m_1,m_2}^1(N_0^2\Delta/d^2)/\Gamma_0(N) \, \stackrel{\beta_{m_1,m_2}^{d,\varrho}}{\longrightarrow} \, \mathcal Q^1(\Delta/d^2)/\Gamma_0(1) \, \longleftarrow \, \mathcal Q_{Np,\varrho,m_1,m_2}^1(N_0^2\Delta/d^2)/\Gamma_0(Np)
 \]
 are bijections, where the arrow on the right can be both $[aNp,b,c] \mapsto [apM_1,b,cM_2]$ or $[aNp,b,c] \mapsto [aM_1,b,cpM_2]$. Indeed, both maps are allowed because $p$ does not divide $m_{\varrho'} = m_{\varrho}$. The proposition follows by applying the same argument for each possible pair $(m_1,m_2)$.
\end{proof}

Suppose now that $p^2 \mid \Delta$. Then we can write a $\Gamma_0(Np)$-invariant decomposition
\[
\mathcal L_{Np}(N_0^2\Delta) = \mathcal L_{Np}^{(p)}(N_0^2\Delta)  \sqcup \mathcal L_{Np}^{[p]}(N_0^2\Delta),
\]
where
\begin{align*}
 \mathcal L_{Np}^{(p)}(N_0^2\Delta) & = \bigsqcup_{\substack{d^2\mid \Delta, \, p \mid d, \\ \mathrm{gcd}(d,N_0)=1}} \bigsqcup_{\varrho' \in R_{Np}(N_0^2\Delta/d^2)} d \cdot \mathcal L_{Np,\varrho'}^1(N_0^2\Delta/d^2), \\
 \mathcal L_{Np}^{[p]}(N_0^2\Delta) & = \bigsqcup_{\substack{d^2\mid \Delta, \, p \nmid d\\ \mathrm{gcd}(d,N_0)=1}} \bigsqcup_{\varrho' \in R_{Np}(N_0^2\Delta/d^2)} d \cdot \mathcal L_{Np,\varrho'}^1(N_0^2\Delta/d^2).
\end{align*}
The first piece consists of `non-$p$-primitive', and we will deal easily with them. We must rather focus our attention now on $\mathcal L_{Np}^{[p]}(N_0^2\Delta)$. To do so, fix a positive integer $d$ with $d^2 \mid \Delta$, $p\nmid d$, and $\mathrm{gcd}(d,N_0)=1$. Observe first that $p$ divides $N_0^2\Delta/d^2$; this already implies that $R_{Np}(N_0^2\Delta/d^2)$ is non-empty whenever so is $R_{N}(N_0^2\Delta/d)$. In addition, observe that if $[aNp,b,c]$ belongs to $\mathcal Q_{Np,\varrho'}^1(\Delta/d^2)$ for some $\varrho' \in R_{Np}(N_0^2\Delta/d^2)$, then either
\begin{itemize}
 \item[i)] $p \mid a$ and $p \nmid c$, or
 \item[ii)] $p\nmid a$ and $p\mid c$.
\end{itemize}
Indeed, this is clear since $p^2$ divides the discriminant of $[aNp,b,c]$ and $p \nmid \mathrm{gcd}(a,b,c)$. We write $\mathcal Q_{Np,\varrho'}^{1,a}(N_0^2\Delta/d^2)$, resp. $\mathcal Q_{Np,\varrho'}^{1,c}(N_0^2\Delta/d^2)$, for the subset of forms in $\mathcal Q_{Np,\varrho'}^1(N_0^2\Delta/d^2)$ falling in case i), resp. ii). In a natural way, proceeding like this for all $d$ and all $\varrho' \in R_{Np}(N_0^2\Delta/d^2)$ (and all the appropriate pairs $(m_1',m_2')$, with the extra condition $\mathrm{gcd}(m_2',N_0)=1$), one eventually defines $\mathcal L_{Np}^{[p],a}(N_0^2\Delta)$ and $\mathcal L_{Np}^{[p],c}(N_0^2\Delta)$, so that 
\[
\mathcal L_{Np}^{[p]}(N_0^2\Delta) = \mathcal L_{Np}^{[p],a}(N_0^2\Delta) \sqcup \mathcal L_{Np}^{[p],c}(N_0^2\Delta).
\]
We will now compare each of the sets $\mathcal L_{Np}^{[p],\star}(N_0^2\Delta)/\Gamma_0(Np)$ with the set $\mathcal L_N^{[p]}(N_0^2\Delta)/\Gamma_0(N)$, for $\star \in \{a,c\}$. 

Fix an integer $d$ and a pair $\varrho', \varrho$ as in lemma \ref{lemma:mrho}. Each admissible factorization $m_{\varrho} = m_1m_2$ gives rise to two possible admissible factorizations $m_{\varrho'} = m_1'm_2'$, namely 
\[
 m_1' = pm_1, m_2' = m_2 \qquad \text{and} \qquad m_1' = m_1, m_2' = pm_2.
\]
With this in mind, we can apply Proposition \ref{prop:formsGKZ} to deduce the following:

\begin{corollary}\label{cor:p2}
 Let $N$, $N_0$, $\Delta$ and $p$ be as above, with $p^2\mid \Delta$. Then the map $\mathrm{id}\colon [aNp,b,c] \mapsto [(ap)N,b,c]$ induces a bijection 
 \[
  \mathcal L_{Np}^{[p],a}(N_0^2\Delta)/\Gamma_0(Np) \, \stackrel{1:1}{\longrightarrow} \, \mathcal L_N^{[p]}(N_0^2\Delta)/\Gamma_0(N),
 \]
 and the map $\tau\colon [aNp,b,c] \mapsto [aN,b,cp]$ induces a bijection 
 \[
  \mathcal L_{Np}^{[p],c}(N_0^2\Delta)/\Gamma_0(Np) \, \stackrel{1:1}{\longrightarrow} \, \mathcal L_N^{[p]}(N_0^2\Delta)/\Gamma_0(N).
 \]
\end{corollary}
\begin{proof}
 Fix a positive integer $d$ with $d^2 \mid \Delta$, $p\nmid d$, and $\mathrm{gcd}(d,N_0)=1$. Fix also an element $\varrho \in R_N(N_0^2\Delta/d^2)$, let $\varrho' \in R_{Np}(N_0^2\Delta/d^2)$ be the unique element such that $r_d(\varrho') = \varrho$, and let $m_{\varrho} = m_1m_2$ be an admissible factorization. By virtue of Proposition \ref{prop:formsGKZ}, 
 both arrows in
 \[
  \mathcal L_{N,\varrho,m_1,m_2}^1(N_0^2\Delta/d^2)/\Gamma_0(N) \, \stackrel{\beta_{m_1,m_2}^{d,\varrho}}{\longrightarrow} \, \mathcal L^1(N_0^2\Delta/d^2)/\Gamma_0(1) \, \longleftarrow \, \mathcal L_{Np,\varrho',pm_1,m_2}^1(N_0^2\Delta/d^2)/\Gamma_0(Np), 
 \]
 are bijections, where the arrow on the right is $[aNp,b,c] \mapsto [apM_1,b,cM_2]$, and similarly both arrows in 
 \[
  \mathcal L_{N,\varrho,m_1,m_2}^1(N_0^2\Delta/d^2)/\Gamma_0(N) \, \stackrel{\beta_{m_1,m_2}^{d,\varrho}}{\longrightarrow} \, \mathcal L^1(N_0^2\Delta/d^2)/\Gamma_0(1) \, \longleftarrow \, \mathcal L_{Np,\varrho',m_1,pm_2}^1(N_0^2\Delta/d^2)/\Gamma_0(Np), 
 \]
 are bijections, where now the arrow on the right is $[aNp,b,c] \mapsto [aM_1,b,cpM_2]$. Noticing that 
 \[
  \mathcal L_{Np,\varrho',pm_1,m_2}^1(N_0^2\Delta/d^2) = \mathcal L_{Np,\varrho',pm_1,m_2}^{1,a}(N_0^2\Delta/d^2)
 \]
 and 
 \[
  \mathcal L_{Np,\varrho',m_1,pm_2}^1(N_0^2\Delta/d^2) = \mathcal L_{Np,\varrho',m_1,pm_2}^{1,c}(N_0^2\Delta/d^2),
 \]
 and applying the same argument for each pair $(m_1,m_2)$, shows that $\mathrm{id}\colon [aNp,b,c] \mapsto [(ap)N,b,c]$ and $\tau\colon [aNp,b,c] \mapsto [aN,b,cp]$ induce the bijections of the statement.
\end{proof}

\subsection{$\d$-th Shintani lifting of $p$-old forms}\label{sec:p-stab}

Fix an odd integer $N \geq 1$, an integer $k \geq 1$, and a Dirichlet character $\chi$ modulo $N$ as before. Let also $p$ be an odd prime not dividing $N$, $f \in S_{2k}^{new}(N, \chi^2)$ be a newform, and consider the $f$-isotypic subspace $S_{2k}(Np,\chi^2)[f] \subset S_{2k}(Np,\chi^2)$. This is the two-dimensional subspace spanned by $f$ and $V_pf$, where $V_p\colon S_{2k}(N,\chi^2) \to S_{2k}(Np,\chi^2)$ denotes the operator given by $V_pf (z) = f(pz)$. To describe the $\d$-th Shintani lifting of an arbitrary form in $S_{2k}(Np,\chi^2)[f]$, it therefore suffices to compute $\theta_{k,Np,\chi,\d}(f)$ and $\theta_{k,Np,\chi,\d}(V_pf)$. If $\alpha = \alpha_p(f)$ and $\beta = \beta_p(f)$ are the roots of the $p$-th Hecke polynomial of $f$, the subspace $S_{2k}(Np,\chi^2)[f]$ is also spanned by the forms
\[
 f_{\alpha} = f - \beta V_pf, \quad f_{\beta} = f - \alpha V_p f,
\]
on which the Atkin--Lehner operator $U_p$ acts as multiplication by $\alpha$ and $\beta$, respectively. We will be later interested in the $\d$-th Shintani lifting of $f_{\alpha}$, under the assumption that $f$ is ordinary at $p$ and $\alpha$ is chosen to be the $p$-unit root of the Hecke polynomial.

Continue to use $N_0$, $N_1$ $\chi_0$, $\epsilon$ with the same meaning and assumptions as in Section \ref{Sec:Shintani-Lifting}, and fix a fundamental discriminant $\d$ satisfying the following conditions:
\begin{itemize}
\item[(i)] $\mathrm{gcd}(N_0,\d) = 1$, $\epsilon(-1)^k\d > 0$;
\item[(ii)] $\d \equiv 0 \pmod p$;
\item[(iii)] $\theta_{k,N,\chi,\d}(f)\neq 0$.
\end{itemize}

Notice that we need to assume $\mathrm{gcd}(N_0,\d) = 1$ for the construction of the $\d$-th Shintani lifting.

\begin{remark}
With our working assumptions, the existence of such a discriminant $\d$ is guaranteed by non-vanishing results for special values of twisted $L$-series (combine, for example, \cite{BFH} and \cite[Th\'eor\`eme 4]{Waldspurger91}) together with Corollary \ref{cor:ad-Lvalue}. 
\end{remark}

Having made this choice for $\d$, we will compute now the $\d$-th Shintani liftings $\theta_{k,Np,\chi,\d}(f)$ and $\theta_{k,Np,\chi,\d}(V_pf)$, and express them in terms of $\theta_{k,N,\chi,\d}(f)$. We start dealing with $\theta_{k,Np,\chi,\d}(f)$.

Let $n\geq 1$ be an integer such that $\epsilon (-1)^k n \equiv 0,1 \pmod 4$. By definition of the $\d$-th Shintani lifting, recall first that the $n$-th Fourier coefficient of $\theta_{k,N,\chi,\d}(f)$ is given by the expression
\begin{align*}
a_n(\theta_{k,N,\chi,\d}(f)) & = C(k,\chi,\d) \sum_{0<t\mid N} \mu(t)\chi_{\d}\bar{\chi}_0(t)t^{-k-1}r_{k,Nt,\chi}(f;\d,\epsilon(-1)^knt^2) = \\
& = C(k,\chi,\d) r_{k,N,\chi}(f;\d,\epsilon(-1)^kn),
\end{align*}
where the second equality is due to the fact that $f$ is new (see equation \eqref{dShintani-withcharacter-fnew}). Similarly, the $n$-th Fourier coefficient of $\theta_{k,Np,\chi,\d}(f)$ is given by  
\[
 a_n(\theta_{k,Np,\chi,\d}(f)) = C(k,\chi,\d)\sum_{0<t\mid Np} \mu(t)\chi_{\d}\bar{\chi}_0(t)t^{-k-1}r_{k,Npt,\chi}(f;\d,\epsilon(-1)^knt^2).
\]
Using that $p$ does not divide $N$, that $\mu(pt) = -\mu(t)$, and that $\chi_{\d}\bar{\chi}_0(pt) = \chi_{\d}\bar{\chi}_0(p)\chi_{\d}\bar{\chi}_0(t)$, we may rewrite the sum in the previous expression as
\[
\sum_{0 < t \mid N} \mu(t)\chi_{\d}\bar{\chi}_0(t)t^{-k-1}\left(r_{k,Npt,\chi}(f;\d,\epsilon(-1)^knt^2) - \chi_{\d}\bar{\chi}_0(p)p^{-k-1}r_{k,Np^2t,\chi}(f;\d,\epsilon(-1)^knp^2t^2)\right).
\]
Our choice of $\d$ implies that $\chi_{\d}(p)=0$, so that we actually have 
\begin{equation}\label{anNp}
a_n(\theta_{k,Np,\chi,\d}(f)) = C(k,\chi,\d)\sum_{0 < t \mid N} \mu(t)\chi_{\d}\bar{\chi}_0(t)t^{-k-1} r_{k,Npt,\chi}(f;\d,\epsilon(-1)^knt^2).
\end{equation}

However, the fact that $f \in S_{2k}^{new}(N,\chi^2)$ is new implies that all the terms in \eqref{anNp} corresponding to non-trivial divisors of $N$ vanish. Indeed:

\begin{lemma}\label{lemma:fonlyt=1}
With the above notation and assumptions, the $\d$-th Shintani lifting of $f$ at level $Np$ satisfies
\[
a_n(\theta_{k,Np,\chi,\d}(f)) = C(k,\chi,\d) \cdot  r_{k,Np,\chi}(f;\d,\epsilon(-1)^kn).
\]
\end{lemma}
\begin{proof}
In view of \eqref{anNp}, it suffices to show that $r_{k,Npt,\chi}(f;\d,\epsilon(-1)^knt^2) = 0$ for all divisors $t$ of $N$ with $t>1$. From \eqref{fkN-rkN} (with $N_1p$ in place of $N_1$) and \eqref{fkN-relation-t} we have 
\[
r_{k,Npt,\chi}(f;\d,\epsilon(-1)^knt^2) =  i_{Npt} \pi^{-1}\binom{2k-2}{k-1}^{-1}2^{2k-2}(|\d|nt^2)^{k-1/2} \cdot \langle f, V_t f_{N_0,N_1p/t}^{k,\bar{\chi}_0}(-;\d, \epsilon(-1)^kn)\rangle,
\]
and we observe that
\[
\langle f, V_t f_{N_0,N_1p/t}^{k,\bar{\chi}_0}(-;\d, \epsilon(-1)^kn)\rangle = \langle f, \mathrm{tr}^{Np}_N(V_t f_{N_0,N_1p/t}^{k,\bar{\chi}_0}(-;\d, \epsilon(-1)^kn))\rangle,
\]
where $\mathrm{tr}^{Np}_N\colon S_{2k}(Np,\chi^2) \to S_{2k}(N,\chi^2)$ is the usual trace operator. Furthermore, one can check that 
\[
\mathrm{tr}^{Np}_N(V_t f_{N_0,N_1p/t}^{k,\bar{\chi}_0}(-;\d, \epsilon(-1)^kn)) = V_t(\mathrm{tr}^{Np/t}_{N/t}f_{N_0,N_1p/t}^{k,\bar{\chi}_0}(-;\d, \epsilon(-1)^kn)),
\]
and hence the Petersson product on the right hand side of the above identity vanishes. Hence, the terms $r_{k,Npt,\chi}(f;\d,\epsilon(-1)^knt^2)$ vanish when $t > 1$ as we wanted to prove.
\end{proof}

Thanks to this lemma, it suffices to compute the quantities $r_{k,Np,\chi}(f;\d,\epsilon(-1)^kn)$ in order to determine the Fourier coefficients $a_n(\theta_{k,Np,\chi,\d}(f))$ of $\theta_{k,Np,\chi,\d}(f)$. Notice that the coefficient $a_n(\theta_{k,Np,\chi,\d}(f))$ is forced to vanish unless $n$ satisfies $\epsilon(-1)^k n \equiv 0,1 \pmod 4$, because $\theta_{k,Np,\chi,\d}(f)$ belongs to Kohnen's plus subspace. Changing slightly the notation, we must compute the $|\dd|$-th Fourier coefficients of $\theta_{k,Np,\chi,\d}(f)$, where $\dd$ is any discriminant with $\d\dd>0$ (note that this is equivalent to $\epsilon(-1)^k\dd > 0$, and that therefore one has $\epsilon(-1)^k|\dd| = \dd$). 

\begin{proposition}\label{prop:thetaN=thetaNp}
 Let $\dd$ be a discriminant such that $\d\dd>0$. Then we have
 \[
    a_{|\dd|}(\theta_{k,Np,\chi,\d}(f)) = a_{|\dd|}(\theta_{k,N,\chi,\d}(f)).
\]
In particular, it follows that $\theta_{k,Np,\chi,\d}(f) = \theta_{k,N,\chi,\d}(f)$. 
\end{proposition}
\begin{proof}
By the previous lemma, this is equivalent to show that $r_{k,Np,\chi}(f;\d,\dd) = r_{k,N,\chi}(f;\d,\dd)$, so it suffices to compute 
\[
    r_{k,Np,\chi}(f;\d,\dd) = \sum_{Q \in \mathcal L_{Np}(N_0^2\d\dd)/\Gamma_0(Np)} \omega_{\d}(Q)I_{k,\chi}(f,Q).
\]
Suppose first that $p \nmid \dd$, so that $p$ divides exactly $N_0^2\d\dd$. By Corollary \ref{cor:nop2}, we can use the identity map $[aNp,b,c] \mapsto [(ap)N,b,c]$ to induce a bijection from $\mathcal L_{Np}(N_0^2\d\dd)/\Gamma_0(Np)$ to $\mathcal L_{N}(N_0^2\d\dd)/\Gamma_0(N)$. Therefore, we have 
\[
r_{k,Np,\chi}(f;\d,\dd) = \sum_{Q \in \mathcal L_{N}(N_0^2\d\dd)/\Gamma_0(N)} \omega_{\d}(Q)I_{k,\chi}(f,Q) = r_{k,N,\chi}(f;\d,\dd),
\]
and it follows that $a_{|\dd|}(\theta_{k,Np,\chi,\d}(f)) = a_{|\dd|}(\theta_{k,N,\chi,\d}(f))$.

Suppose now that $p \mid \dd$, hence $p^2$ divides $\d\dd$. In this case, we use the decomposition 
\[
 \mathcal L_{Np}(N_0^2\d \dd)/\Gamma_0(Np) = \mathcal L_{Np}^{(p)}(N_0^2\d \dd)/\Gamma_0(Np) \sqcup \mathcal L_{Np}^{[p]}(N_0^2\d \dd)/\Gamma_0(Np)
\]
already introduced in the previous paragraph. Since $\omega_{\d}(Q) = 0$ for all forms $Q \in \mathcal L_{Np}^{(p)}(N_0^2\d \dd)$, we see that the sum over the first set does not contribute to $r_{k,Np,\chi}(f;\d,\dd)$. Therefore, 
\[
r_{k,Np,\chi}(f;\d,\dd) = \sum_{Q \in \mathcal L_{Np}^{[p],a}(N_0^2\d\dd)/\Gamma_0(Np)} \omega_{\d}(Q)I_{k,\chi}(f,Q) + \sum_{Q \in \mathcal L_{Np}^{[p],c}(N_0^2\d\dd)/\Gamma_0(Np)} \omega_{\d}(Q)I_{k,\chi}(f,Q),
\]
using the notations from the previous paragraph. Applying now Corollary \ref{cor:p2}, the second sum equals 
\[
	\sum_{Q \in \mathcal L_{N}^{[p]}(N_0^2\d\dd)/\Gamma_0(N)}\omega_{\d}(\tau^{-1}(Q))I_{k,\chi}(f,\tau^{-1}(Q)) = 0.
\]
Since $\omega_{\d}(\tau^{-1}(Q))= \chi_{\dd}(p)\omega_{\d}(Q) = 0$ for all $Q \in \mathcal L_{N}^{[p]}(N_0^2\d\dd)$, the above expression simplifies to 
\begin{equation}\label{eqn:Np-a}
	r_{k,Np,\chi}(f;\d,\dd) = \sum_{Q \in \mathcal L_{Np}^{[p],a}(N_0^2\d\dd)/\Gamma_0(Np)} \omega_{\d}(Q)I_{k,\chi}(f,Q) = \sum_{Q \in \mathcal L_{N}^{[p]}(N_0^2\d\dd)/\Gamma_0(N)} \omega_{\d}(Q)I_{k,\chi}(f,Q),
\end{equation}
where the last equality holds because of Corollary \ref{cor:p2}. This sum equals $r_{k,N,\chi}(f;\d,\dd)$, since we can replace $\mathcal L_{N}^{[p]}(N_0^2\d\dd)$ by $\mathcal L_{N}(N_0^2\d\dd)$ using again that $\omega_{\d}(Q) = 0$ for all $Q \in \mathcal L_{N}^{(p)}(N_0^2\d\dd)$, which implies the result. 
\end{proof}

Since $f$ is an arbitrary newform, from the above proposition we directly obtain the following:

\begin{corollary} \label{cor:compatibility} The following diagram commutes:
\[
    \xymatrix{
        S_{2k}^{new}(N,\chi^2) \ar@{^{(}->}[rr] \ar[d]^{\theta_{k,N,\chi,\d}} & & S_{2k}(Np,\chi^2) \ar[d]^{\theta_{k,Np,\chi,\d}} \\
        S_{k+1/1}^{+,new}(N,\chi) \ar@{^{(}->}[rr] & &S_{k+1/2}^+(Np,\chi). 
    }
\]
\end{corollary}

\begin{corollary}
 If $\dd$ is a discriminant such that $\d\dd > 0$, then 
 \[
  a_{|\dd|}(\theta_{k,Np,\chi,\d}(V_pf)) =  p^{-k}\bar{\chi}(p)\left(\frac{\dd}{p}\right) a_{|\dd|}(\theta_{k,N,\chi,\d}(f)) + a_{|\dd|/p^2}(\theta_{k,N,\chi,\d}(f)),
 \]
 where $a_{|\dd|/p^2}(\theta_{k,N,\chi,\d}(f)) = 0$ when $p^2\nmid \d'$. In particular, if $p^2\nmid \dd$, this reduces to
 \[
  a_{|\dd|}(\theta_{k,Np,\chi,\d}(V_pf)) =  p^{-k}\bar{\chi}(p)\left(\frac{\dd}{p}\right) a_{|\dd|}(\theta_{k,N,\chi,\d}(f)).
 \]
\end{corollary}

\begin{proof}
First of all, since $f \in S_{2k}^{new}(N,\chi^2)$ it is well-known that 
\[
 T_p f = U_p f + \chi^2(p) p^{2k-1}V_p f.
\]
Applying the Hecke-equivariance of the $\d$-th Shintani lifting and the commutative diagram of Corollary \ref{cor:compatibility}, it follows that 
\[
	T(p^2) \theta_{k,N,\chi,\d}(V_pf) = U(p^2) \theta_{k,N,\chi,\d}(V_pf)+ \chi^2(p) p^{2k-1} \theta_{k,Np,\chi,\d}(V_pf),
\]
where $T(p^2) = T_{k+1/2,N,\chi}(p^2)$ and $U(p^2)$ are the operators defined at the end of Section \ref{sec:notation}. Using the explicit expression for $T(p^2)$ in equation \eqref{eqn:Tp-Shimura} and $U(p^2)$, one deduces the result.
\end{proof}

For later reference, let us collect in the following corollary the explicit expression for the $\d$-th Shintani lifting of $f_{\alpha} = f - \beta V_pf$.

\begin{corollary}\label{cor:lift-falpha}
 If $\dd$ is a discriminant such that $\d\dd > 0$, then 
 \[
  a_{|\dd|}(\theta_{k,Np,\chi,\d}(f_{\alpha})) = (1 - \beta \chi_{\dd}\bar{\chi}(p) p^{-k})a_{|\dd|}(\theta_{k,N,\chi,\d}(f)) - \beta a_{|\dd|/p^2}(\theta_{k,N,\chi,\d}(f)).
 \]
 In particular, if $p^2\nmid \dd$, then 
 \[
  a_{|\dd|}(\theta_{k,Np,\chi,\d}(f_{\alpha})) = (1 - \beta \chi_{\dd}\bar{\chi}(p) p^{-k})a_{|\dd|}(\theta_{k,N,\chi,\d}(f)).
 \]
\end{corollary}

\begin{remark} \label{rmk:coeff-pstabilization}
 Since $a_{|\dd|}(\theta_{k,N,\chi,\d}(f)) = a_{|\dd|}(\theta_{k,Np,\chi,\d}(f))$, if $p^2 \nmid \dd$ we can rewrite the last identity in Corollary \ref{cor:lift-falpha} as
\[
     a_{|\dd|}(\theta_{k,Np,\chi,\d}(f_{\alpha})) = (1-\beta \chi_{\dd}\bar{\chi}(p)p^{-k})a_{|\dd|}(\theta_{k,Np,\chi,\d}(f)).
\]
In this form, this formula holds true also when $f \in S_{2k}^{new}(Np,\chi^2)$ is new of level $Np$ (hence its ordinary $p$-stabilization is $f$ itself, and $\beta = 0$).

\end{remark}

\begin{remark}
The above corollary should be read as complementary to \cite[Proposition 2.10]{Mak17}, as he imposes the condition $\mathrm{gcd}(Np,\d)=1$ whereas we choose $\d$ such that $\mathrm{gcd}(N,\d)=1$ and $\d \equiv 0 \pmod p$. 
\end{remark}

\section{Modular symbols and their interpolation}  \label{Sec:p-adic}

As a preparation for the next section, we now review some basic material on the $p$-adic interpolation of the modular symbols, and set the conventions of Hida theory that will be used later. In most of the discussion, we follow closely and adapt the approach in \cite{GS93, Stevens}.

\subsection{Modular symbols}

Let $M\geq 1$ be an integer and let $\Gamma_0(M)$. Let $\Delta^0 := \Div^0(\P^1(\Q))$ be the group of degree 0 divisors on the set of rational cusps of Poincar\'e's upper half plane $\mathfrak H$. The congruence group $\Gamma_0(M)$ acts by linear fractional transformations on $\Delta^0$. Let $V$ be a right $\Z[1/6][\Gamma_0(M)]$-module. There is a natural right action of $\Gamma_0(M)$ on the set $\mathrm{Hom}(\Delta^0,V)$ of additive homomorphisms from $\Delta^0$ to $V$, defined by the rule
\[
(\Phi|\gamma)(D) := \Phi(\gamma D)|\gamma, \qquad \gamma\in \Gamma_0(M), \, D \in \Delta^0.
\]
The group of $V$-valued modular symbols over $\Gamma_0(M)$ is the set of $\Phi:\Delta^0\to V$ such that $\Phi | \gamma = \Phi$ for all $\gamma \in \Gamma_0(M)$, i.e.:
\[
\Symb_{\Gamma_0(M)}(V) := \mathrm{Hom}_{\Gamma_0(M)}(\Delta^0,V) = \mathrm{Hom}(\Delta^0,V)^{\Gamma_0(M)}.
\]

Let now $G_0(M)$ denote the semigroup of two-by-two matrices $\gamma = (\begin{smallmatrix}a & b \\ c & d\end{smallmatrix}) \in M_2(\Z)$ such that $\mathrm{gcd}(a,M)=1$ and $c\equiv 0 \pmod M$. If the action of $\Gamma_0(M)$ on $V$ extends to an action of $G_0(M)$, then $\Symb_{\Gamma_0(M)}(V)$ inherits a natural action of Hecke operators.

The element $\iota = \mathrm{diag}(1,-1) \in G_0(M)$ induces an involution on $\Symb_{\Gamma_0(M)}(V)$. Notice that $\iota$ acts by $\Phi \mapsto \Phi|\iota$, with $(\Phi|\iota)(D) = \Phi(\iota D)|\iota$. Under the assumption that $V$ is a $\Z[1/2]$-module, the action of $\iota$ decomposes the modular symbol $\Phi$ as $\Phi = \Phi^+ + \Phi^-$, where $\Phi^{\pm}:=\frac{1}{2}(\Phi\pm \iota \Phi) \in \Symb_{\Gamma_0(M)}(V)^{\pm}$ are such that $\iota \Phi^\pm = \pm \Phi^\pm$, i.e.:
\[
\Symb_{\Gamma_0(M)}(V) = \Symb_{\Gamma_0(M)}(V)^+ \oplus \Symb_{\Gamma_0(M)}(V)^-.
\]

%=====V = L_r=====
Let us now focus our discussion on some special choices of $V$. Fix an integer $r\geq 0$ and a commutative ring $R$ in which $2r!$ is invertible (e.g. of characteristic zero). Let $\Sym^r(R^2)$ denote the $R$-module of homogeneous `divided powers polynomials' of degree $r$ in $X$, $Y$ over $R$ generated by the monomials
\[
\frac{X^n}{n!}\frac{Y^{r-n}}{(r-n)!} \quad \text{with } 0 \leq n \leq r.
\]
Similarly, let $\Sym^r(R^2)^*$ be the $R$-module of homogeneous polynomials of degree $r$ in $X$, $Y$ over $R$, generated by the monomials $X^n Y^{r-n}$ with $0 \leq n\leq r$. Both $\Sym^r(R^2)$ and $\Sym^r(R^2)^*$ are equipped with a natural action of $M_2(R)$, by the rule 
\[
(F|\gamma)(X,Y) := F((X,Y){}^t\gamma)
\]
where $\gamma \mapsto {}^t\gamma$ denotes the usual transpose.

Let us write $\langle \cdot, \cdot \rangle_r$ for the unique perfect pairing 
\[
\langle \cdot, \cdot \rangle_r\colon \Sym^r(R^2)\times \Sym^r(R^2)^* \, \longrightarrow \, R
\]
satisfying 
\[
\left\langle  \frac{X^iY^{r-i}}{i!(r-i)!}, X^{r-j}Y^{j} \right\rangle_{ \hspace{-3pt} r} = (-1)^j \delta_{ij},
\]
where $\delta_{ij}$ is the usual Kronecker's delta function. This pairing satisfies in addition the properties 
\[
\left\langle  \frac{(aY-bX)^r}{r!}, P(X,Y) \right\rangle_{\hspace{-3pt} r} = P(a,b) \, \, \, \forall \, a,b \in R, \quad \langle P_1|\gamma, P_2|\gamma \rangle_r = \det(\gamma)^r\langle P_1,P_2\rangle.
\]

From now on, we will write $L_r(R)$ (resp. $L_r^*(R)$) for the $R[\Gamma_0(M)]$-module $\Sym^r(R^2)$ (resp. $\Sym^r(R^2)^*$) equipped with the action of $\Gamma_0(M)$ induced by the above described action. If in addition $\chi$ is an $R$-valued Dirichlet character modulo $M$, we denote by $L_{r,\chi}(R)$ (resp. $L_{r,\chi}^*(R)$) the same underlying $R[\Gamma_0(M)]$-module as $L_r(R)$ (resp. $L_r^*(R)$) but with the action of $\Gamma_0(M)$ twisted by $\chi^{-1}$ (resp. $\chi$). That is, for an element
\[
\gamma = \left(\begin{array}{cc} a & b \\ c & d \end{array}\right) \in \Gamma_0(M)
\]
one has 
\[
(F|\gamma)(X,Y) := \begin{cases}
\chi(a) \cdot F((X,Y){}^t\gamma) & \text{if } F \in L_{r,\chi}(R),\\
\chi(d) \cdot F((X,Y){}^t\gamma) & \text{if } F \in L_{r,\chi}^*(R).
\end{cases}
\]
If $R$ is an $R'$-algebra, we say that a modular symbol $\varphi \in  \Symb_{\Gamma_0(M)}(L_{r,\chi}(R))$ is defined over $R'$ if it takes values in  $L_{r,\chi}(R')$, i.e. if it lies in the image of the natural map
\[
\Symb_{\Gamma_0(M)}(L_{r,\chi}(R')) \hookrightarrow  \Symb_{\Gamma_0(M)}(L_{r,\chi}(R)).
\]

If $f \in S_{k}(M,\chi)$ is a cusp form of weight $k$, level $\Gamma_0(M)$, and Nebentype character $\chi$, then the $L_{k-2,\chi}(\C)$-valued differential form 
\[
\omega_f := \frac{1}{(k-2)!} f(\tau)(\tau Y - X)^{k-2} d\tau
\]
on $\mathfrak H$ satisfies $\gamma^* \omega_f|\gamma = \omega_f$ for all $\gamma \in \Gamma_0(M)$. The additive map $\psi_f\colon \Delta^0 \to L_{k-2,\chi}(\C)$ induced by 
\[
\{c_2\} - \{c_1\} \, \longmapsto \, \int_{c_1}^{c_2} \omega_f
\]
yields an $L_{k-2,\chi}(\C)$-valued modular symbol over $\Gamma_0(M)$ (where the integral is along the oriented geodesic in $\mathfrak H$ from $c_1$ to $c_2$). We then have the following:

\begin{theorem} 
For each choice of sign $\pm$, the map $f \mapsto \psi_f$ yields a Hecke-equivariant inclusion 
\begin{equation} \label{ES}
S_{k}(M,\chi) \, \hookrightarrow \, \Symb_{\Gamma_0(M)}(L_{k-2,\chi}(\C))^{\pm}.
\end{equation}
\end{theorem}

\begin{proof}
	This is a combination of the Eichler--Shimura isomorphism, the Manin--Drinfeld principle and the Ash--Stevens isomorphism (see \cite[Proposition 4.2]{AshStevens}), which composed give us maps
	\[
		S_{k}(M,\chi) \stackrel[\text{ES}]{\sim}{\longrightarrow} H^1_\mathrm{par}(\Gamma_0(M), L_{k-2,\chi}(\C))^\pm \stackrel[\text{DM}]{}{\hookrightarrow} H^1_c(\Gamma_0(M), L_{k-2,\chi}(\C))^\pm  \stackrel[\textrm{AS}]{\sim}{\longrightarrow}  \Symb_{\Gamma_0(M)}(L_{k-2,\chi}(\C))^\pm.
	\]
	We want to stress that the first isomorphism comes from integration and it does not respect algebraicity. If $f$ is defined over a subring $R$ of $\C$ its image is not necessarily in $H^1_\mathrm{par}(\Gamma_0(M), L_{k-2,\chi}(R))^\pm$. Also, the inclusion provided by the Manin--Drinfeld principle arises by taking a section of the natural projection map $H^1_c(\Gamma_0(M), L_{k-2,\chi}(\C)) \to H^1_\mathrm{par}(\Gamma_0(M), L_{k-2,\chi}(\C))$. Such a section is defined over any characteristic zero field, but it does not need to descend to any subring $R$ of $\C$.
\end{proof}

If $f \in S_{k}(M,\chi)$ is a Hecke eigenform and $\mathcal O_f$ denotes the ring of integers of its Hecke field, it is well-known that there exist two complex numbers $\Omega_f^{\pm} \in \C^{\times}$ such that the normalized modular symbols
\begin{equation}\label{shimuraperiodsf}
\varphi_f^{\pm} := \frac{1}{\Omega_f^{\pm}} \cdot \psi_f^{\pm} \in \mathrm{Symb}_{\Gamma_0(M)}(L_{k-2,\chi}(\C))^{\pm}, 
\end{equation}
one for each choice of sign, are defined over $\mathcal O_f$ (cf. \cite{ManinPeriods}).

\subsection{Hida theory} \label{hidasection}

Let $\cO$ be a finite extension of $\Zp$ and let $\Gamma := 1+p\Zp$. We write $\Lambda = \Lambda_\cO = \cO[[\Gamma]]$ for the usual Iwasawa algebra and consider the space:
\[
    \mathcal X(\Lambda) := \Hom_{\cO-\cont}(\Lambda, \barQp).
\]
Elements $\mathbf a \in \Lambda$ can be seen as functions on $\mathcal X(\Lambda)$ through evaluation at $\mathbf a$, i.e. by setting $\mathbf a(\kappa) := \kappa(\mathbf a)$. The set $\mathcal X(\Lambda)$ is endowed with the analytic structure induced from the natural identification 
\begin{equation}\label{XLambda-characters}
\mathcal X(\Lambda) \simeq \mathrm{Hom}_\cont(\Gamma, \bar{\Q}_p^{\times})
\end{equation}
between $\mathcal X(\Lambda)$ and the group of continuous characters $\kappa\colon \Gamma \to \bar{\Q}_p^{\times}$. We define the subset of {\em classical} points (or characters) as
\[
    \mathcal X^{\mathrm{cl}}(\Lambda) := \{ [t]\mapsto t^{k-2}\mid k\geq 2 \}.
\]

If $\mathcal R$ is a finite flat $\Lambda$-algebra then we write 
\[
    \mathcal X(\mathcal R) := \mathrm{Hom}_{\cO-\cont}(\mathcal R, \bar{\Q}_p)
\]
for the set of continuous $\cO$-homomorphisms from $\mathcal R$ to $\bar{\Q}_p$, to which we also refer as `points of $\mathcal R$'. The restriction to $\Lambda$ (via the structure morphism $\Lambda \to \mathcal R$) induces a surjective finite-to-one map 
\[
    \pi\colon \mathcal X(\mathcal R) \, \longrightarrow \, \mathcal X(\Lambda).
\]
A point $\kappa \in \mathcal X(\mathcal R)$ is said to be classical if the point $\pi(\kappa) \in \mathcal X(\Lambda)$ is classical, meaning that for all $[t]\in \Gamma \subset \Lambda$, $\kappa([t])=t^{k-2}$, for some $k>2$. We will write $\mathcal X^{\mathrm{cl}}(\mathcal R)$ for the subset of arithmetic points in $\mathcal X(\mathcal R)$.
One can define analytic charts around all points $\kappa$ of $\mathcal X(\mathcal R)$ which are unramified over $\Lambda$, by building sections $S_\kappa$ of the map $\pi$, so that $\mathcal X(\mathcal R)$ inherits the structure of rigid analytic cover of $\mathcal X(\Lambda)$.
A function $f\colon \mathcal U \subseteq \mathcal X(\mathcal R) \to \bar{\Q}_p$ defined on an analytic neighborhood of $\kappa$ is analytic if so is $f \circ S_{\kappa}$. The evaluation at an element $\mathbf a \in \mathcal R$ yields a function $\mathbf a\colon \mathcal X(\mathcal R) \to \bar{\Q}_p$, $\mathbf a(\kappa) := \kappa(\mathbf a)$, which is analytic at every unramified point of $\mathcal X (\mathcal R)$.

If $N$ is a positive integer, with $p \nmid N$, consider the $\Lambda$-algebra
\[
\Lambda_N := \cO[[\Z_{p,N}^{\times}]] \simeq \Lambda[\Delta_{Np}]
\]
associated with the completed group ring on $\Z_{p,N}^{\times} := \varprojlim (\Z/Np^m\Z)^{\times}$. Under the natural isomorphisms 
\[
    \Z_{p,N}^{\times} \simeq \Z_p^{\times} \times \Delta_N \simeq \Gamma \times \Delta_{Np}, \quad \text{where } \Delta_{M} := (\Z/M\Z)^{\times},
\]
we will write $t_p \in \Z_p^{\times}$ and $t_N \in \Delta_N$ for the projections of $t \in \Z_{p,N}^{\times}$ in $\Z_p^{\times}$ and $\Delta_N^{\times}$. We will further write $\langle t_p\rangle \in \Gamma$ for the projection of $t_p$ in $\Gamma$. Observe that $t_p = \langle t_p \rangle \omega(t_p)$, where $\omega$ denotes the Teichmuller character. Notice also that the analytic space $\mathcal X(\Lambda_N)$ is naturally isomorphic to a product of $\varphi(Np)$ copies of $\mathcal X(\Lambda)$, with the components being in one-to-one correspondence with the $\bar{\Q}_p$-valued characters $\psi$ of $\Delta_{Np}$. The inclusion of the $\psi$-component in the full space is described on rings by the map $\mathrm{loc}_\psi\colon \Lambda_N\to \Lambda$ given by $[t] \mapsto \psi(t)[\langle t\rangle]$ on group-like elements.

Fix a prime $p \geq 5$, and an integer $N\geq 1$ such that $p \nmid N$. If $k\geq 2$ and $m\geq 1$ are integers, recall that an eigenform $f \in S_{k}(\Gamma_1(Np^m),\bar{\Q}_p)$ is said to be {\em ordinary} (at $p$) if the eigenvalue $a_p$ of $T_p$ acting on $f$ is a $p$-adic unit. If $f$ is also normalized ($a_1=1$) and the prime-to-$p$ part of the conductor is $N$, one then says that $f$ is a {\em $p$-stabilized newform} of tame conductor $N$. One can check that if $f$ is an ordinary $p$-stabilized newform of tame conductor $N$, then either $f$ is already a newform, or $f$ is related to a newform $g$ of conductor $N$ by the so-called process of ordinary $p$-stabilization (and in this case the level of $f$ is $Np$). In the second case, if $\alpha$ and $\beta$ denote the roots of the $p$-th Hecke polynomial for $g$, labelled so that $\alpha$ is the unit root and $\beta$ is the non-unit root, then one has $f(z) = g(z) - \beta g(pz)$. We write $S_{k}^\ord(\Gamma_1(Np^m),\bar{\Q}_p)$ for the set of ordinary $p$-stabilized newforms in $S_{k}(\Gamma_1(Np^m),\bar{\Q}_p)$. Assume from now on that $\cO$ contains the values of the characters of $\Delta_{Np}$.

\begin{definition}\label{def:Hidafamily}
    A {\em Hida family} of tame level $N$ and tame character $\chi$ modulo $N$ is a quadruple $(\mathcal R_{\mathbf f}, \cU_{\mathbf f}, \cU_{\mathbf f}^\mathrm{cl}, \mathbf f)$ where:
    \begin{itemize}
        \item $\mathcal R_{\mathbf f}$ is a finite flat integral domain extension of $\Lambda$;
        \item $\mathcal U_{\mathbf f} \subset \mathcal X(\mathcal R_{\mathbf f})$ is an open subset for the rigid analytic topology; 
        \item $\cU_{\mathbf f}^{\mathrm{cl}} \subset \mathcal U_{\mathbf f} \cap \mathcal X^{\mathrm{cl}}(\mathcal R_{\mathbf f})$ is a dense subset of $\mathcal U_{\mathbf f}$ whose weights are contained in a single residue class $k_0-2$ modulo $p-1$; 
        \item $\mathbf f \in \mathcal R_{\mathbf f}[[q]]$ is a power series in $q$ with coefficients in $\mathcal R_{\mathbf f}$, such that for all $\kappa\in \cU^{\mathrm{cl}}_{\mathbf f}$ of weight $k-2> 0$, 
        \[
            \mathbf f(\kappa) \in S_{k}^{\mathrm{ord}}(Np,\chi)
        \]
        is the ordinary $p$-stabilization of a newform $f_\kappa \in S_{k}(N,\chi)$.
\end{itemize}
\end{definition}

\begin{remark} 
When $k-2=0$, the form $\mathbf f(\kappa)$ can be either old or new at $p$. Only in the first case, $\mathbf f(\kappa)$ would be the $p$-stabilization of a weight 2 newform of level $N$.
\end{remark}

\begin{theorem}[Hida Theory]
    Let $f\in S_{k}^{\mathrm{ord}}(Np,\chi)$ be a $p$-stabilized newform. Then there exists a unique Hida family $(\mathcal R_{\hf},\mathcal U_{\hf}, \mathcal U_{\hf}^{\mathrm{cl}},\hf)$ and a unique point $\kappa\in \mathcal U_{\hf}^{\mathrm{cl}}$, such that $\hf_{\kappa}=f$.
\end{theorem}

\begin{remark} \label{rmk:UniversalVersion}
    A Hida family can be seen as a branch of the big Hecke $\Lambda_N$-algebra $\mathcal R_N$ built by Hida (see \cite{Hid86}). In fact, by the universal property of $\mathcal R_N$, there exists a unique $\Lambda$-algebra homomorphism $\mathrm{loc}_{\mathbf f}\colon \mathcal R_N \to \mathcal R_{\mathbf f}$, which gives $\mathcal R_{\mathbf f}$ the structure of $\Lambda_N$-algebra determined on group like elements $[t]\in \Lambda_N$ by  $\mathrm{loc}_\hf([t]) = \mathrm{loc}_{\chi\omega^{k_0-2}}([t]) = \chi(t)\omega^{k_0-2}(t)[\langle t \rangle]$.
\end{remark}

\subsection{$p$-adic interpolation of modular symbols}\label{sec:padicMS}

The modular symbols associated with modular forms as recalled above can be $p$-adically interpolated, giving rise to `$\Lambda$-adic modular symbols'. We now recall this construction, mainly due to Greenberg and Stevens \cite{GS93}.

Consider the subset $(\Z_p^2)'$ of {\em primitive} vectors in $\Z_p^2$, meaning the subset of vectors which do not lie in $(p\Z_p)^2$, and let $\mathbf D = \mathrm{Meas}((\Z_p^2)')$ denote the group of $\Z_p$-valued measures on $(\Z_p^2)'$. Namely, if $\mathrm{Cont}((\Z_p^2)')$ denotes the space of continuous $\Z_p$-valued functions on $(\Z_p^2)'$, then $\mathbf D$ is the space of continuous $\Z_p$-valued functionals on $\mathrm{Cont}((\Z_p^2)')$. One can also see $\mathbf D$ as a direct summand of $\mathrm{Meas}(\Z_p^2)$, via restriction of continuous functions on $\Z_p^2$ to $(\Z_p^2)'$. Following standard conventions of measure theory, if $\mu \in \mathbf D$ we will write 
\[
\int_U f d\mu
\]
to denote $\mu(f \cdot \mathbf 1_U)$ for any $f \in \mathrm{Cont}((\Z_p^2)')$ and any compact open subset $U \subseteq (\Z_p^2)'$.

There are various natural actions on $\mathbf D$. On the one hand, the scalar action of $\Z_p^{\times}$ on $(\Z_p^2)'$ induces a natural action of $\Z_p[[\Z_p^{\times}]]$ on $\mathbf D$. On the other hand, viewing elements of $(\Z_p^2)'$ as row vectors, we let $\Gamma_0(N)$ act on $(\Z_p^2)'$ by multiplication on the right. This action induces a natural (right) action on $\mathbf D$, which is characterized by the fact that, for all $\mu \in \mathbf D$, $\gamma \in \Gamma_0(N)$, and $f \in \mathrm{Cont}((\Z_p^2)')$, one has 
\[
\int f(x,y) d(\mu|\gamma)(x,y) = \int f((x,y){}^t\gamma^{-1})d\mu(x,y).
\]
The two actions just described clearly commute one with each other, hence we can view $\mathbf D$ as a $\Z_p[[\Z_p^{\times}]][\Gamma_0(N)]$-module.

If $\mathcal R$ is a $\Lambda_N$-algebra, we define the $\mathcal R[\Gamma_0(N)]$-module
\[
\mathbf D_{\mathcal R} := \mathbf D \otimes_{\Z_p[[\Z_p^{\times}]]} \mathcal R,
\]
where $\Gamma_0(N)$ acts through the rule \[
(\mu \otimes \lambda) | \gamma := \mu|\gamma \otimes [a]_N\lambda,
\]
for $\mu \in \mathbf D$, $\lambda \in \mathcal R$, and $\gamma \in \Gamma_0(N)$, where $a$ is the upper-left entry of $\gamma$ and $[a]_N \in \Delta_N = (\Z/N\Z)^{\times} \subseteq \Lambda_N \to \mathcal R$ is the image in $\mathcal R$ of the group-like element of $a$ modulo $N$. %To lighten the notation, we will write $\mathbf D_N = \mathbf D_{\Lambda_N}$. 

Let $(\mathcal R_{\mathbf f}, \mathcal U_{\mathbf f}, \mathcal U_{\mathbf f}^{\mathrm{cl}}, \mathbf f)$ be a Hida family of tame level $N$ and tame character $\chi$ modulo $N$, as in Definition \ref{def:Hidafamily}. We shall consider the $\mathcal R_\hf[\Gamma_0]$-module $\mathbf D_{\mathbf f} := \mathbf D_{\mathcal R_{\mathbf f}}$. If $\kappa \in \mathcal U_{\mathbf f}^{\mathrm{cl}}$ has weight $k-2 \geq 0$, so that $\mathbf f(\kappa) \in S_{k}^{\mathrm{ord}}(Np,\chi)$, then we have a specialization map
\[
    \mathrm{sp}_{\kappa}(\mu\otimes\alpha) = \kappa(\alpha) \cdot \int_{\Z_p\times\Z_p^{\times}} \frac{(xY-yX)^{k-2}}{r!}d\mu(x,y)
\]
for $\mu \in \mathbf D$, $\alpha \in \mathcal R_\hf$, which yields a specialization map on modular symbols
\begin{align*}
    \mathrm{sp}_{\kappa,*}\colon \mathrm{Symb}_{\Gamma_0(N)}(\mathbf D_{\mathbf f}) \, &\to \, \mathrm{Symb}_{\Gamma_0(Np)}(L_{k-2,\chi}(\mathcal O_{\hf(\kappa)})) \\
    \Phi &\mapsto \Phi_\kappa:=\mathrm{sp}_{\kappa,*}(\Phi).
\end{align*}
For simplicity, we may write $\psi_{\mathbf f (\kappa)}^{\pm}$ for the $\pm$-components of the modular symbol associated with $\mathbf f(\kappa)$. For each classical point, we may fix complex periods $\Omega_{\mathbf f(\kappa)}^{\pm} \in \C^{\times}$ so that the normalized cohomology classes
\begin{equation}\label{ManinperiodsfN}
\varphi_{\hf(\kappa)}^{\pm} := \frac{1}{\Omega_{\hf(\kappa)}^{\pm}} \psi_{\hf(\kappa)}^\pm
\end{equation}
are defined over the ring of integers $\mathcal O_{\hf(\kappa)}$ of the Hecke field of $\hf(\kappa)$ (cf. \eqref{shimuraperiodsf}). Following \cite{GS93}, we recall how the collection of cohomology classes $\varphi_{\hf(\kappa)}^{\pm}$, as $\kappa$ varies on the classical points of $\mathcal R_\hf$, can be put together into $\Lambda$-adic cohomology classes (or modular symbols) with the following

\begin{theorem}\label{thm:Lambda-adic-MS}
Let $\kappa_0 \in \mathcal U_{\mathbf f}^{\mathrm{cl}} \subseteq \mathcal X^{\mathrm{cl}}(\mathcal R_{\mathbf f})$ be a classical point. There exists a modular symbol $\Phi_{\mathbf f} \in \mathrm{Symb}_{\Gamma_0(N)}(\mathbf D_{\mathbf f})$, and a choice of $p$-adic periods $\Omega_{\kappa} \in R_{\kappa}$, one for each $\kappa \in \mathcal U_{\mathbf f}^{\mathrm{cl}}$, such that:
\begin{itemize}
    \item[i)] $\Omega_{\kappa_0} \neq 0$;
    \item[ii)] $\Phi_{\mathbf f,\kappa} = \Omega_{\kappa}\cdot \varphi_{\mathbf f(\kappa)}^-$ for all $\kappa \in \mathcal U_{\mathbf f}$.
\end{itemize}
\end{theorem}

\begin{proof}
The proof is essentially a consequence of \cite[Theorem 5.13]{GS93}, as explained in \cite[Theorem 5.5]{Stevens}. We recall some details here for convenience of the reader. The modular symbol $\varphi_{\hf(\kappa_0)}^-$ is a Hecke eigenclass, so that \cite[Theorem 5.13]{GS93} tells us that the $\mathcal R_\hf$-module of Hecke eigenclasses in the space $\mathrm{Symb}_{\Gamma_0(N)}(\mathbf D_{\mathcal R_\hf})^-$ is free of rank one, and it is generated by an element $\Psi$ whose image under $\mathrm{sp}_{\kappa_0,*}$ equals $\varphi_{\hf(\kappa_0)}^-$. One can choose an element $\alpha \in\mathcal R_\hf$ such that $\alpha(\kappa_0) \neq 0$ and $\alpha \Psi$ is everywhere regular. Then $\Phi = \alpha\Psi$ is the desired modular symbol. Indeed, by weak multiplicity one, the weight $\kappa$ specialization $\Phi_{\kappa}$ is a multiple of $\varphi_{\hf(\kappa)}^-$ for each arithmetic point $\kappa$, thus one can choose periods $\Omega_{\kappa} \in R_{\kappa}$ verifying conditions i) and ii).
\end{proof}

\section{The \texorpdfstring{$\Lambda$}{Lambda}-adic \texorpdfstring{$\mathfrak d$}{d}-th Shintani lifting} \label{Sec:Lambda-adic}

This section is devoted to the construction of the so-called `$\Lambda$-adic $\d$-th Shintani lifting', which interpolates the $\d$-th Shintani liftings of modular forms in $p$-adic families. To provide this construction, first we explain in \ref{sec:cohomological} a cohomological interpretation of the classical $\d$-th Shintani lifting described in Section \ref{Sec:Shintani-Lifting}, which is better suited for the $p$-adic interpolation and yields also an algebraicity statement. After that, in Section \ref{sec:LambdaShintani} we refine Stevens' approach in \cite{Stevens} to define the $\Lambda$-adic $\d$-th Shintani lifting $\Theta_{\d}(\mathbf f)$ of a Hida family $\mathbf f$ (cf. Equation \eqref{def:LambdaShintani}). The interpolation property is precisely stated in Theorem \ref{thm:LambdaShintani-interpolation}. The main result of Section \ref{sec:LambdaKohnen} is the $\Lambda$-adic version of Kohnen's formula given in Theorem \ref{thm:comparisonGS}. As an application of this, Corollary \ref{cor:KohnenFormula-Np} gives a mild generalization of Kohnen's classical formula. Finally, in Section \ref{sec:Shadows} we study the evaluation of the $\Lambda$-adic Kohnen formula at classical points outside the interpolation region.

\subsection{Cohomological $\d$-th Shintani lifting and integrality}\label{sec:cohomological}

Let $M\geq 1$ be an odd integer, and $\chi$ be a Dirichlet character modulo $M$ (later we will be interested in $M=N$ or $Np$). Let $\chi_0$ be the primitive character associated with $M$, $M_0$ be its conductor, and set $M_1 = M/M_0$, $\epsilon = \chi(-1)$. Let $k\geq 1$ be an integer, and fix a fundamental discriminant $\d$ with $\mathrm{gcd}(M_0,\d)=1$ and $\epsilon (-1)^k\d > 0$. We explain the cohomological interpretation of the $\d$-th Shintani lifting from $S_{2k}(M,\chi^2)$ to $S_{k+1/2}^+(M,\chi)$.

If $f \in S_{2k}(M,\chi^2)$, recall that the definition of its $\d$-th Shintani lifting $\theta_{k,M,\d}(f)$ involves certain integrals $I_{k,\chi}(f,Q)$ associated with integral binary quadratic forms $Q$ of discriminant divisible by $|\d|$ (see \eqref{dShintani-withcharacter} and \eqref{rkN-def}). Namely, if $m\geq 1$ is an integer such that $\epsilon(-1)^km \equiv 0,1\pmod 4$, then the $m$-th Fourier coefficient of $\theta_{k,M,\chi,\d}(f)$ involves the computation of integrals
\[
I_{k,\chi}(f,Q) = \chi_0(Q) \int_{C_Q} f(z)Q(z,1)^{k-1}dz
\]
for integral binary quadratic forms $Q$ in $\mathcal L_{Mt}(|\d|mt^2)$ (modulo $\Gamma_0(Mt)$-equivalence), where $t$ is a positive divisor of $M$. If $Q$ is such a quadratic form, we let \[
D_Q := \partial C_Q = \{\omega_Q'\} - \{\omega_Q\} \in \Delta^0
\]
be the degree zero divisor given by the boundary of the geodesic path $C_Q$.

\begin{definition}
Let $R$ be a $\Z[1/6]$-algebra containing the values of $\chi$, $\varphi \in \Symb_{\Gamma_0(M)}(L_{2k-2,\chi^2}(R))$ be a modular symbol, and $t > 0$ be a divisor of $M$. For each integral binary quadratic form $Q \in \mathcal L_{Mt}$ with positive discriminant, we put 
\[
	J_{k,\chi}(\varphi,Q) := \chi_0(Q)\langle \varphi(D_Q),Q^{k-1}\rangle \in R,
\]
where here $\langle \, \rangle = \langle \, \rangle_{2k-2}$ is the pairing on modular symbols as defined in the previous section. For each integer $m\geq 1$ with $\epsilon(-1)^km\equiv 0,1 \pmod 4$ we also define  
\[
s_{k,Mt,\chi}(\varphi,\d; \epsilon(-1)^km) := \sum_{Q \in \mathcal L_{Mt}(|\d|m)/\Gamma_0(Mt)} \omega_\d(Q) \cdot J_{k,\chi}(\varphi,Q).
\]
\end{definition}

In the above definition, $\varphi(D_Q)$ stands for the value at $D_Q$ of any cocycle representing $\varphi$. One can check that $J_{k,\chi}(\varphi,Q)$ does not depend on the choice of such representative for $\varphi$, and that it depends on $Q \in \mathcal L_{Mt}$ only up to $\Gamma_0(Mt)$-equivalence. When $\chi$ is trivial, we will write $J_k$ instead of $J_{k,\chi}$, and similarly $s_{k,Mt}$ instead of $s_{k,Mt,\chi}$. With the help of the quantities $s_{k,Mt,\chi}(\varphi,\d; \epsilon(-1)^km)$, we can now define the cohomological $\d$-th Shintani lifting as follows (compare with the classical definition in \eqref{dShintani-withcharacter}).

\begin{definition}
Let $R$ be a $\Z[1/6]$-algebra containing the values of $\chi$. Define an $R$-linear map 
\[
\Theta_{k,M,\chi,\d}\colon \Symb_{\Gamma_0(M)}(L_{2k-2}(R)) \, \longrightarrow \, R[[q]]
\]
by setting
\[
\Theta_{k,M,\chi,\d}(\varphi) := C(k,\chi,\d) \sum_{\substack{m\geq 1,\\\epsilon(-1)^km\equiv 0,1 (4)}} \left(\sum_{0 < t|M_1} \mu(t)\chi_{\d}\bar{\chi}_0(t)t^{-k-1} s_{k,Mt,\chi}(\varphi,\d; \epsilon(-1)^kmt^2)\right) q^m.
\]
\end{definition}

\begin{proposition}\label{prop:cohomologicalShintani}
Let the notation be as before, and $R$ be a $\Z[1/6]$-algebra.
\begin{itemize}
    \item[i)] For every $\varphi \in \Symb_{\Gamma_0(M)}(L_{2k-2,\chi^2}(R))$, one has $\Theta_{k,M,\chi,\d}(\varphi|\iota) = - \Theta_{k,M,\chi,\d}(\varphi)$.
    \item[ii)] Let $f \in S_{2k}(M,\chi^2)$, and let $\psi_f \in  \Symb_{\Gamma_0(M)}(L_{2k-2,\chi^2}(\C))$ be its associated modular symbol as above. Then $\Theta_{k,M,\chi,\d}(\psi_f) = \Theta_{k,M,\chi,\d}(\psi_f^-) = \theta_{k,M,\chi,\d}(f)$.
    \item[iii)] If $R = K$ is a field of characteristic zero, and $\varphi \in  \Symb_{\Gamma_0(M)}(L_{2k-2,\chi^2}(K))$ belongs to the image of the inclusion \eqref{ES}, then $\Theta_{k,M,\chi,\d}(\varphi) \in S_{k+1/2}^+(M,\chi;K)$. 
\end{itemize}
\end{proposition}

\begin{proof}
First of all, one easily checks from the definitions that $\iota D_Q = -D_{Q\circ\iota}$. This implies that, for an arbitrary integral binary quadratic form $Q$,
\begin{align*}
    J_{k,\chi}(\varphi|\iota,Q) & = \chi_0(Q)\langle \varphi(\iota D_Q)|\iota, Q^{k-1}\rangle = \chi_0(Q) \langle \varphi(\iota D_Q),(Q|\iota)^{k-1}\rangle = \\
    & = - \chi_0(Q\circ\iota)\langle \varphi(D_{Q\circ\iota}),(Q\circ\iota)^{k-1}\rangle = -J_{k,\chi}(\varphi,Q\circ\iota),
\end{align*}
where in the third equality we use that $Q \circ \iota = Q | \iota$ and that $\chi_0(Q) = \chi_0(Q\circ \iota)$. Statement i) follows from this by the definition of $\Theta_{k,M,\chi,\d}$ (noticing that also $\omega_{\d}(Q) = \omega_{\d}(Q\circ \iota)$, since $Q$ and $Q\circ \iota$ represent the same integers, and that $Q \mapsto Q\circ \iota$ gives bijections on each of the sets $\mathcal L_{Mt}(|\d|m)/\Gamma_0(Mt)$).
    
Next, we observe that    
\[
\langle (\tau Y - X)^{2k-2}, Q^{k-1}\rangle = (2k-2)!\cdot Q(\tau,1)^{k-1}.
\]
This implies that 
\[
I_{k,\chi}(f,Q) = \frac{\chi_0(Q)}{(2k-2)!} \int_{C_Q} f(\tau) \langle (\tau Y-X)^{2k-2},Q^{k-1}\rangle d\tau = \chi_0(Q) \langle \psi_f(D_Q),Q^{k-1}\rangle = J_{k,\chi}(\psi_f, Q),
\]
and hence one easily sees that $\Theta_{k,M,\chi,\d}(\psi_f) = \theta_{k,M,\chi,\d}(f)$. Furthermore, decomposing $\psi_f = \psi_f^+ + \psi_f^-$ into its $+$ and $-$ components, part i) tells us that $\Theta_{k,M,\chi,\d}(\psi_f^+) = 0$, and hence $\Theta_{k,M,\chi,\d}(\psi_f) = \Theta_{k,M,\chi,\d}(\psi_f^-)$, thereby completing the proof of ii). Finally, statement iii) follows already from the proof of ii).
\end{proof}

A cohomology class (or modular symbol) $\Phi \in H^0_c(\Gamma_0(M),L_{2k-2}(\C))$ is said to be {\em defined over} a subring $R$ of $\C$ if $\Phi$ lies in the image of the natural map 
\[
H^0_c(\Gamma_0(M),L_{2k-2}(R)) \, \longrightarrow \, H^0_c(\Gamma_0(M),L_{2k-2}(\C)).
\]
Equivalently, $\Phi$ is defined over $R$ if the corresponding modular symbol (under Ash--Stevens isomorphism) takes values in $L_{2k-2}(R)$.

If $f \in S_{2k}(M,\chi^2)$, recall from \eqref{shimuraperiodsf} that there exist complex numbers $\Omega_f^{\pm} \in \C^{\times}$ such that the normalized modular symbols
\[
\varphi_f^{\pm} := \frac{1}{\Omega_f^{\pm}} \cdot \psi_f^{\pm}, 
\]
are defined over the ring of integers $\mathcal O_f$ of the Hecke field of $f$. We fix once and for all periods $\Omega_f^{\pm} \in \C^{\times}$ with this algebraicity property, and then define  
\begin{equation} \label{eqn:theta-alg}
    \theta^{\mathrm{alg}}_{k,M,\chi,\d}(f) := \frac{1}{\Omega_f^-}\cdot \theta_{k,M,\chi,\d}(f).
\end{equation}

\begin{theorem}\label{thm:theta-alg-ms}
Let $f \in S_{2k}(M,\chi^2)$ be a Hecke eigenform, and let $\psi_f$ and $\varphi_f$ be as above. Then 
\[
\theta^{\mathrm{alg}}_{k,M,\chi,\d}(f) = \Theta_{k,M,\chi,\d}(\varphi_f^-) \in S_{k+1/2}^+(M,\chi;\mathcal O_f).
\]
\end{theorem}
\begin{proof}
The equality $\theta^{\mathrm{alg}}_{k,M,\chi,\d}(f) = \Theta_{k,M,\chi,\d}(\varphi_f^-)$ follows from the previous proposition and the definition of $\varphi_f^-$ and $\theta_{k,M,\chi,\d}^{\mathrm{alg}}(f)$. Now, since $\varphi_f^-$ is defined over $\mathcal O_f$, the values $\chi_0(Q)\langle \varphi_f^-(D_Q), Q^{k-1} \rangle$ belong to $\mathcal O_f$ for all $Q$, hence also the values $s_{k,Mt,\chi}(\varphi_f^-,\d;(-1)^kmt^2)$ belong to $\mathcal O_f$ for all $m$, $t$. As a consequence, all the coefficients of $\Theta_{k,M,\chi,\d}(\varphi_f^-)$ lie in $\mathcal O_f$ and the theorem is proved.
\end{proof}

As a direct consequence of Corollary \ref{cor:lift-falpha}, and of the cohomological $\d$-th Shintani lifting introduced in this section, we have the following statement.

\begin{corollary}
Let $k\geq 1$ be an integer, $N \geq 1$ be an odd integer, $\chi$ be a Dirichlet character modulo $N$, and $p>2$ be a prime with $p\nmid N$. Let $f \in S_{2k}^{new}(N,\chi^2)$ be a normalized newform ordinary at $p$, and let $f_{\alpha} \in S_{2k}(Np,\chi^2)$ be its ordinary $p$-stabilization. Let $\d$ be a fundamental discriminant with $\epsilon(-1)^k\d > 0$, $\mathrm{gcd}(\d,N)=1$ and $\d \equiv 0 \pmod p$, and let $\dd$ be a discriminant such that $\d\dd > 0$ and $p\nmid \dd$. Then 
\[
a_{|\dd|}(\Theta_{k,Np,\chi,\d}(\psi_{f_{\alpha}})) = (1- \beta \chi_{\dd}\bar{\chi}_0(p)p^{-k})a_{|\dd|}(\theta_{k,N,\chi,\d}(f)),
\]
and hence
\[
\frac{1}{\Omega_f^-} \cdot a_{|\dd|}(\Theta_{k,Np,\chi,\d}(\psi_{f_{\alpha}})) = (1- \beta \chi_{\dd}\bar{\chi}_0(p)p^{-k}) a_{|\dd|}(\theta_{k,N,\chi,\d}^{\mathrm{alg}}(f)).
\]
\end{corollary}
\begin{proof}
Indeed, we have 
\[
a_{|\dd|}(\Theta_{k,Np,\chi,\d}(\psi_{f_{\alpha}})) = a_{|\dd|}(\theta_{k,Np,\chi,\d}(f_{\alpha})) = (1 - \beta \chi_{\dd}\bar{\chi}_0(p)p^{-k}) a_{|\dd|}(\theta_{k,N,\chi,\d}(f)),
\]
where the first equality follows from part ii) of Proposition \ref{prop:cohomologicalShintani}, and the second one from Corollary \ref{cor:lift-falpha}. This proves the first identity. The second one follows immediately from the first by the definition of $\theta_{k,N,\chi,\d}^{\mathrm{alg}}$.
\end{proof}

\subsection{The $\d$-th Shintani lifting of a Hida family} \label{sec:LambdaShintani}

Fix now a Hida family $(\mathcal R_{\mathbf f}, \mathcal U_{\mathbf f},\mathcal U_{\mathbf f}^{\mathrm{cl}},\mathbf f)$ of tame level $N$ and tame character $\chi^2$ modulo $N$. As usual, we write $\chi_0$ be the primitive character associated with $\chi$, $N_0$ for its conductor and $N_1 = N/N_0$. We assume that $\mathrm{gcd}(N_0,N_1)=1$ and write $\epsilon = \chi(-1)$. By definition of Hida family, there exists an integer $r_0$ such that every classical point $\kappa \in \mathcal U_{\mathbf f}^{\mathrm{cl}}$ has weight $2k-2$ with $2k \equiv 2r_0 \pmod{p-1}$. Notice that this integer is uniquely determined only modulo $(p-1)/2$. We define 
\[
    \widetilde{\mathcal R}_{\mathbf f} := \mathcal R_{\mathbf f} \otimes_{\Lambda,\sigma} \Lambda, 
\]
where $\sigma\colon \Lambda \to \Lambda$ is the $\mathcal O$-algebra isomorphism induced by $[t] \mapsto [t^2]$ on $1+p\Z_p$. In particular, notice that $\widetilde{\mathcal R}_{\mathbf f}$ is isomorphic to $\mathcal R_{\mathbf f}$ as $\mathcal O$-algebras. We equip $\widetilde{\mathcal R}_{\mathbf f}$ with the structure of $\Lambda$-algebra via the map $\lambda \mapsto 1 \otimes \lambda$.  The natural homomorphism 
\[
    \mathcal R_{\mathbf f} \, \longrightarrow \, \widetilde{\mathcal R}_{\mathbf f}, \quad \alpha \, \longmapsto \, \alpha \otimes 1
\]
is an isomorphism of $\mathcal O$-algebras, but it is {\em not} even a homomorphism of $\Lambda$-algebras. Indeed, similarly as above this is reflected in the fact that the induced map 
\[
    \pi\colon \mathcal X(\widetilde{\mathcal R}_{\mathbf f}) \, \longmapsto \, \mathcal X(\mathcal R_{\mathbf f})
\] 
on weight spaces doubles the weights.
 
Let $\widetilde{\mathcal U}_{\mathbf f} := \pi^{-1}(\mathcal U_{\mathbf f})$ and $\widetilde{\mathcal U}_{\mathbf f}^{\mathrm{cl}} := \pi^{-1}(\mathcal U_{\mathbf f}^{\mathrm{cl}})$. Notice that not all the weights of classical points in $\widetilde{\mathcal U}_{\mathbf f}^{\mathrm{cl}}$ are contained in a single residue class modulo $p-1$.  Indeed, a point $\tilde{\kappa} \in \widetilde{\mathcal U}_{\mathbf f}^{\mathrm{cl}}$ has weight $k-1$ for some integer $k$ such that $2k\equiv 2r_0 \pmod{p-1}$. Therefore either $k\equiv r_0\pmod{p-1}$ or $k\equiv r_0+\frac{p-1}{2}\pmod{p-1}$, yielding a partition of $\widetilde{\mathcal U}_{\mathbf f}^{\mathrm{cl}}$ as a union of two sets. In the following discussion we write $\widetilde{\mathcal U}_{\mathbf f}^{\mathrm{cl}}(r_0)$ for the subset of classical points in $\widetilde{\mathcal U}_{\mathbf f}^{\mathrm{cl}}$ whose weights are of the form $k-1$ with $k \equiv r_0 \pmod{p-1}$.

\begin{remark} \label{rmk:classicaltilde}
    It is worth to explain the meaning of the above choice, in light of the universal ordinary $p$-adic Hecke algebra $\mathcal R_N$ discussed in Remark \ref{rmk:UniversalVersion}. Indeed, we can consider its metaplectic cover $\widetilde{\mathcal R}_N := \mathcal R_N \otimes_{\Lambda_N,\sigma} \Lambda_N$, 
    where the tensor product is taken with respect to the $\mathcal O$-algebra homomorphism $\sigma\colon \Lambda_N \to \Lambda_N$ induced by the group homomorphism $t \mapsto t^2$ on $\Z_{p,N}^{\times}$. For any choice of square root $\tilde\chi$ of $\chi^2\omega^{2r_0-2}$ in $\widehat\Delta_{Np}$ we have a $\Lambda_N$-algebra structure for $\widetilde{\mathcal R}_{\mathbf f}$. This choice uniquely determines the $\widetilde{\mathcal R}_N$-algebra structure of $\mathcal R_\hf$ because of the following diagram:
    \[
        \xymatrix{
                &\mathcal R_N    \ar[rr]^{\mathrm{loc}_{\mathbf f}} && \mathcal R_{\mathbf{f}} \\
            \Lambda_N \ar[ur] \ar[dr]_\sigma \ar[rr]^{\mathrm{loc}_{\chi^2\omega^{2r_0-2}}} && \Lambda \ar[ur] \ar[dr]^{\sigma} \\
            & \Lambda_N \ar[rr]_{\mathrm{loc}_{\tilde\chi}} && \Lambda 
        }
    \]
    This amounts to the choice of a single irreducible component $\mathcal X(\mathcal R_\hf)$ of the bigger space $\mathcal X(\mathcal R_N)$. 
    
    Since $p\nmid N$, we have a natural decomposition $\widehat\Delta_{Np} \simeq \widehat\Delta_N\times \widehat\Delta_p$, hence $\tilde\chi$ is uniquely determined by the choice of a square root of $\chi^2$ in $\widehat\Delta_N$ and a square root of $\omega^{2r_0-2}$ in $\widehat\Delta_p$. In our setting $\chi$ is already chosen in the classical context, so we are left with the choice of $r_0$.  Notice that our choice of integer $r_0$ determines a class modulo $p-1$ which corresponds to the $\Lambda_N$-structure of $\widetilde{\mathcal R}_{\mathbf f}$ induced by the choice $\tilde\chi := \chi\omega^{r_0-1}$.
\end{remark}

\begin{lemma}\label{lemma:JQ}
    Let $(\mathcal R_{\mathbf f},\mathcal U_{\mathbf f}, \mathcal U_{\mathbf f}^{\mathrm{cl}}, \mathbf f)$ be a Hida family and keep the notation as in the above discussion. For each $Q \in \mathcal L_{Np}^{[p]}$ with positive discriminant divisible by $p$, there is a unique $\mathcal R_{\mathbf f}$-homomorphism 
    \[
        J_{Q,\mathbf f}^{r_0}\colon \mathbf D_{\mathbf f} \, \longrightarrow \, \widetilde{\mathcal R}_{\mathbf f}
    \]
    such that for all $\tilde{\kappa} \in \widetilde{\mathcal U}_{\mathbf f}^{\mathrm{cl}}(r_0)$ of weight $k-1$ one has
    \[
        \tilde{\kappa}(J^{r_0}_{Q,\mathbf f}(\mu)) = \chi_0(Q)\langle \mathrm{sp}_{\pi(\tilde\kappa)}(\mu), Q^{k-1} \rangle \quad \text{for all } \mu \in \mathbf D_{\mathbf f}.
    \]
\end{lemma}

\begin{proof}
Fix a quadratic form $Q$ as in the statement. The uniqueness of $J_Q$ is clear because $\widetilde{\mathcal U}_{\mathbf f}^{\mathrm{cl}}(r_0)$ is dense in $\widetilde{\mathcal U}_{\mathbf f}$. To prove the existence, we adapt the arguments in the proof of \cite[Lemma 6.1]{Stevens}.

Recall that there is a canonical isomorphism $\Z_p[[\Z_p^{\times}]] \simeq \mathrm{Meas}(\Z_p^{\times})$, $j \mapsto dj$. Using this, we first define a map 
\[
j_Q\colon \mathbf D \, \longrightarrow \, \Z_p[[\Z_p^{\times}]], \quad \nu \mapsto j_Q(\nu),
\]
where $j_Q(\nu)$ is determined by requiring that 
\[
\int_{\Z_p^{\times}} f \, dj_Q(\nu) = \int_{\Z_p\times\Z_p^{\times}} f(Q(x,y)) \, d\nu(x,y)
\]
for all continuous functions $f\colon \Z_p^{\times} \to \Z_p$. Observe that since $Q \in \mathcal L_{Np}^{[p]}$ has discriminant divisible by $p$, we have $Q(x,y) \in \Z_p^{\times}$ for all $(x,y) \in \Z_p\times\Z_p^{\times}$, so this is well-defined. We notice that $j_Q$ is a $\Z_p$-linear map, and that $j_Q([t]\nu) = [t^2]\cdot j_Q(\nu)$ for $t \in \Z_p^{\times}$. We compose the map $j_Q$ with the map $\mathrm{loc}_{r_0} := \mathrm{loc}_{\tilde\chi}\mid_{\Lambda_1}$, i.e.
\begin{align*}
    \mathrm{loc}_{r_0}\colon \Lambda_1 \, &\to \, \Lambda \\
[x] & \mapsto \omega^{r_0-1}(x) [\langle x\rangle] \quad \text{for } x \in \Z_p^{\times}.
\end{align*}
In this way we get a map
\[
j_{Q,\mathbf f}\colon \mathbf D \, \longrightarrow \Lambda, \quad \nu \, \longmapsto \, j_{Q,\mathbf f}(\nu) = \mathrm{loc}_{r_0}(j_Q(\nu)).
\]
which, in terms of measures, is characterized by requiring that
\[
    \int_{\Gamma} \varphi \, dj_{Q,\mathbf f}(\nu) = \int_{\Z_p^{\times}} \varphi^{\dagger} \, dj_Q(\nu)
\]
for all continuous functions $\varphi\colon \Gamma \to \Z_p$, where $\varphi^{\dagger}(x) = \omega^{r_0-1}(x)\varphi(\langle x \rangle)$ for $x \in \Z_p^{\times}$. Observe that $j_{Q,\mathbf f}([t]\nu)=\mathrm{loc}_{r_0}([t]^2j_Q(\nu))$ for all $t \in \Z_p^{\times}$, $\nu \in \mathbf D$. The map $j_{Q,\mathbf f}$ extends by $\mathcal R_{\mathbf f}$-linearity to a unique $\mathcal R_{\mathbf f}$-linear map 
\[
    J_{Q,\mathbf f}\colon \mathbf D_{\mathbf f} = \mathbf D \otimes_{\Z_p[[\Z_p^{\times}]]} \mathcal R_{\mathbf f}  \, \longrightarrow \, \widetilde{\mathcal R}_{\mathbf f} = \mathcal R_{\mathbf f} \otimes_{\Lambda,\sigma} \Lambda
\]
such that 
\[
J_{Q,\mathbf f}(\nu \otimes \alpha) = \chi_0(Q) \cdot \alpha \otimes_{\sigma} j_{Q,\mathbf f}(\nu) \quad \text{for } \nu \in \mathbf D, \alpha \in \mathcal R_{\mathbf f}.
\]

Now, let $\mu = \nu \otimes \alpha \in \mathbf D_{\mathbf f} = \mathbf D \otimes \mathcal R_{\mathbf f}$, and $\tilde{\kappa} \in \widetilde{\mathcal U}_{\mathbf f}^{\mathrm{cl}}(r_0)$ be a classical point of weight $k-1$, and let $\kappa = \pi(\tilde{\kappa}) \in \mathcal U_{\mathbf f}^{\mathrm{cl}}$. Then we have 
\begin{align*}
    \tilde{\kappa}(J_{Q,\mathbf f}^{r_0}(\mu)) & = \chi_0(Q) \cdot \tilde{\kappa}(\alpha \otimes 1) \tilde{\kappa}(1 \otimes j_{Q,\mathbf f}(\nu)) = \chi_0(Q) \kappa(\alpha) \cdot \int_{\Gamma} \tilde{\kappa} \, dj_{Q,\mathbf f}(\nu) = \\
 &= \chi_0(Q) \kappa(\alpha) \cdot \int_{\Z_p^{\times}} \tilde{\kappa}^{\dagger} \, dj_Q(\nu)
 = \chi_0(Q)\kappa(\alpha)\cdot \int_{\Z_p\times\Z_p^{\times}} \omega^{r_0-1}(Q(x,y))\tilde{\kappa}(\langle Q(x,y) \rangle) d\nu(x,y) = \\
 & = \chi_0(Q) \kappa(\alpha)\cdot \int_{\Z_p\times\Z_p^{\times}} \omega(Q(x,y))^{r_0-k} Q(x,y)^{k-1} d\nu(x,y) = \\
 & = \chi_0(Q) \kappa(\alpha)\cdot \int_{\Z_p\times\Z_p^{\times}}  \left\langle \frac{(xY-yX)^{2k-2}}{(2k-2)!}, Q^{k-1}\right\rangle d\nu(x,y) = \\
 & = \chi_0(Q) \left\langle \kappa(\alpha)\cdot \int_{\Z_p\times\Z_p^{\times}}  \frac{(xY-yX)^{2k-2}}{(2k-2)!} d\nu(x,y), Q^{k-1} \right\rangle = \chi_0(Q)\langle \mathrm{sp}_{\kappa}(\mu), Q^{k-1}\rangle,
\end{align*}
as we wanted to prove.
\end{proof}

\begin{definition}\label{def:JfQ}
Let $\Phi_{\mathbf f} \in \mathrm{Symb}_{\Gamma_0(N)}(\mathbf D_{\mathbf f})$ be the $\Lambda$-adic modular symbol attached to $\mathbf f$ as in Theorem \ref{thm:Lambda-adic-MS}. For each $Q \in \mathcal L_{Np}^{[p]}$ with positive discriminant divisible by $p$, we define
\[
J^{r_0}(\mathbf f; Q) := J_{Q,\mathbf f}^{r_0}(\Phi_{\mathbf f}(D_Q)) \in \widetilde{\mathcal R}_{\mathbf f}.
\] 
\end{definition}

An easy computation shows that $J^{r_0}(\mathbf f; Q)$ depends only on the $\Gamma_0(Np)$-equivalence class of $Q$. The notation suggests that $J^{r_0}(\mathbf f;-)$ depends only on the Hida family $\mathbf f$ (and on the choice of $r_0$), but the definition clearly shows that it depends on the modular symbol $\Phi_{\mathbf f}$, which is only determined up to the choice of $p$-adic periods as explained in Theorem \ref{thm:Lambda-adic-MS}. Despite of this, we prefer to write $J^{r_0}(\mathbf f;-)$ instead of $J^{r_0}(\Phi_{\mathbf f}; -)$.

Now let $\d$ be a fundamental discriminant {\em divisible by $p$} such that $\mathrm{gcd}(N_0,\d)=1$.

\begin{definition}\label{def:sf}
If $m\geq 1$ is an integer, we define 
\[
s_N^{r_0}(\mathbf f; \d, \chi_{\d}(-1) m) := \sum_{Q \in \mathcal L_{Np}^{[p]}(|\d|m)/\Gamma_0(Np)} \omega_{\d}(Q)\cdot J^{r_0}(\mathbf f; Q) \in \widetilde{\mathcal R}_{\mathbf f}.
\]
\end{definition}

Observe that the function $s_N^{r_0}(\mathbf f; \d, \chi_{\d}(-1) m)$ is identically zero if $\chi_{\d}(-1)m \not\equiv 0,1 \pmod 4$, since in this case $|\d|m = \d \chi_{\d}(-1)m$ is not a discriminant. And when $\chi_{\d}(-1)m \equiv 0,1 \pmod 4$, the function $s_N^{r_0}(\mathbf f;\d, \chi_{\d}(-1)m)$ on $\mathcal X(\widetilde{\mathcal R}_{\mathbf f})$ is key to interpolate the $m$-th Fourier coefficients of the $\d$-th Shintani liftings of the specializations of $\mathbf f$. Indeed, we define the $\Lambda$-adic $\d$-th Shintani lifting of $\mathbf f$ as the power series with coefficients in $\widetilde{\mathcal R}_{\mathbf f}$ given by
\begin{equation}\label{def:LambdaShintani}
    \Theta_{\d}^{r_0}(\mathbf f) := \sum_{m\geq 1} \left(\sum_{0<t\mid N_1} \mu(t)\chi_{\d}\bar{\chi}_0(t)t^{-2}\omega(t)^{r_0-1}[\langle t\rangle]^{-1}s_N^{r_0}(\mathbf f; \d, \chi_{\d}(-1) mt^2)\right) q^m \in \widetilde{\mathcal R}_{\mathbf f}[[q]],
\end{equation}
where $[\langle t\rangle] \in \Lambda$ is the group-like element associated with $\langle t \rangle \in \Gamma$ (notice that $p \nmid N_1$). Observe that $\Theta_{\d}^{r_0}(\mathbf f)$ vanishes identically if $\epsilon(-1)^{r_0}\d < 0$ because the quantities $s_N^{r_0}(\mathbf f; \d, \chi_{\d}(-1)mt^2)$ are all zero (compare with Remark \ref{rmk:vanishing}).  Note also that the definition of $\Theta_{\d}^{r_0}(\mathbf f)$ depends on the residue class of $r_0$ modulo $p-1$, as suggested by the notation (in fact, note that Lemma \ref{lemma:JQ} and Definitions \ref{def:JfQ} and \ref{def:sf} have the same dependence). As explained in Remark \ref{rmk:classicaltilde}, we have two possible choices for this residue class, corresponding to different branches in the metaplectic cover, and resulting in two different $\Lambda$-adic liftings. The interpolation property for $\Theta_{\d}^{r_0}(\mathbf f)$ then reads as follows:

\begin{theorem}\label{thm:LambdaShintani-interpolation}
    Keep the notation notation as above. For every classical point $\tilde{\kappa} \in \widetilde{\mathcal U}_{\mathbf f}^{\mathrm{cl}}(r_0)$ of weight $k-1$, we have 
    \[
        \Theta_{\d}^{r_0}(\mathbf f)(\tilde{\kappa}) = C(k,\chi,\d)^{-1}  \Theta_{k,Np,\d}(\Phi_{\mathbf f,\kappa}) = 
        \Omega_{\kappa} \cdot C(k,\chi,\d)^{-1} \theta_{k,Np,\chi,\d}^{\mathrm{alg}}(\mathbf f(\kappa)),
    \]
    where $\kappa = \pi(\tilde{\kappa}) \in \mathcal U_{\mathbf f}^{\mathrm{cl}}$, and $C(k,\chi,\d)$ is the constant defined in \eqref{constant:Ckchid}.
\end{theorem}
\begin{proof}
The proof follows from the above construction of the $\Lambda$-adic $\d$-th Shintani lifting. First suppose that $\epsilon(-1)^{r_0}\d < 0$. In this case, $\epsilon(-1)^k\d < 0$ as well for all points $\tilde{\kappa}$ as in the statement, and hence both sides of the stated equality are zero. We assume for the rest of the proof that $\epsilon(-1)^{r_0}\d > 0$. Then, if $\tilde{\kappa}$ is as in the statement we have $\chi_{\d}(-1) = \epsilon(-1)^k$. If $m\geq 1$ is an integer with $\epsilon(-1)^k m \not\equiv 0,1 \pmod 4$, then $s_N^{r_0}(\mathbf f; \d, \epsilon(-1)^k m) = 0$, since $|\d|m$ is not a discriminant. For the remaining positive integers, we can apply Lemma \ref{lemma:JQ} to find 
\begin{align*}
   \tilde{\kappa}(s_N^{r_0}(\mathbf f,\d; \chi_{\d}(-1) m)) & = \sum_{Q \in \mathcal L_{Np}^{[p]}(|\d|m)/\Gamma_0(Np)} \omega_{\d}(Q)\cdot \tilde{\kappa}(J^{r_0}(\mathbf f; Q)) = \\
   & = \sum_{Q \in \mathcal L_{Np}^{[p]}(|\d|m)/\Gamma_0(Np)} \omega_{\d}(Q)\cdot \langle \mathrm{sp}_{\kappa}(\Phi_{\mathbf f}(D_Q), Q^{k-1})\rangle = \\
   & = \sum_{Q \in \mathcal L_{Np}(|\d|m)/\Gamma_0(Np)} \omega_{\d}(Q)\cdot J_k(\Phi_{\mathbf f,\kappa},Q) = \\
   & =  s_{k,Np,\chi}(\Phi_{\mathbf f,\kappa},\d;\epsilon(-1)^k m).
\end{align*}
In the last line, the presence of $\omega_\d$ makes trivial the contribution of non-$p$-primitive forms (because $p \mid \d$). It then easily follows by the definition of $\Theta_{\d}^{r_0}$ that
\[
    \Theta_{\d}^{r_0}(\mathbf f)(\tilde{\kappa}) = C(k,\chi,\d)^{-1} \Theta_{k,Np,\chi,\d}(\Phi_{\mathbf f,\kappa}),
\]
observing that
\[
\tilde{\kappa}(t^{-2}\omega(t)^{r_0-1}[\langle t \rangle]^{-1}) = t^{-2}\omega(t)^{r_0-k}t^{-k+1} = t^{-k-1}.
\]
The second equality in the statement is now deduced using that $\Phi_{\mathbf f,\kappa} = \Omega_{\kappa} \cdot \varphi_{\mathbf f(\kappa)}^-$, by Theorem \ref{thm:Lambda-adic-MS}, and that $\Theta_{k,Np,\chi,\d}(\varphi_{\mathbf f(\kappa)}^-) = \theta_{k,Np,\chi,\d}^{\mathrm{alg}}(\mathbf f(\kappa))$, by Theorem \ref{thm:theta-alg-ms}.
\end{proof}

\begin{remark}
As we have pointed out above, it may happen that $\Theta_{\d}^{r_0}(\mathbf f)$ vanishes identically. However, suppose that there is at least one classical point $\tilde{\kappa} \in \widetilde{\mathcal U}_{\mathbf f}^{\mathrm{cl}}(r_0)$, say of weight $k-1$, such that the classical $\d$-th Shintani lifting $\theta_{k,Np,\d}(\mathbf f(\kappa))$ is non-zero, where $\kappa = \pi(\tilde{\kappa})$. This implies that $\epsilon(-1)^k\d > 0$, and hence also $\epsilon(-1)^{r_0}\d > 0$. Then, by virtue of Theorem \ref{thm:Lambda-adic-MS} we can choose the $p$-adic periods so that $\Omega_{\kappa} \neq 0$, and the interpolation property of the previous theorem ensures that $\Theta_{\d}^{r_0}(\mathbf f)$ does not vanish. Hence, as soon as the classical $\d$-th Shintani lifting is not vanishing on $\mathbf f(\kappa)$ for some $\kappa \in \pi(\widetilde{\mathcal U}_{\mathbf f}^{\mathrm{cl}}(r_0))$, one can construct a non-zero $\Lambda$-adic $\d$-th Shintani lifting of $\mathbf f$, $\Theta_{\d}^{r_0}(\mathbf f)$. When doing this choice, we can say that $\Theta_{\d}^{r_0}(\mathbf f)$ is `centered at $\kappa$'.
\end{remark}

\subsection{A $\Lambda$-adic Kohnen formula} \label{sec:LambdaKohnen} 

We keep the notation and assumptions as in the previous paragraph. Associated with a Hida family $(\mathcal R_{\mathbf f}, \mathcal U_{\mathbf f}, \mathcal U_{\mathbf f}^{\mathrm{cl}}, \mathbf f)$ and a fundamental discriminant $\d$ satisfying the assumptions of Theorem \ref{thm:LambdaShintani-interpolation}, let us rewrite \eqref{def:LambdaShintani} as
\[
    \Theta_{\d}^{r_0} (\mathbf f) = \sum_{m\geq 1} \mathbf a_m(\Theta_{\d}^{r_0}(\mathbf f)) q^m\in \widetilde{\mathcal R}_{\mathbf f}[[q]].
\] 
According to Theorem \ref{thm:LambdaShintani-interpolation}, the elements $\mathbf a_m(\Theta_{\d}^{r_0}(\mathbf f)) \in \widetilde{\mathcal R}_{\mathbf f}$ interpolate Fourier coefficients of the $\d$-th Shintani liftings of the classical forms $\mathbf f(\kappa)$ on the set of classical points $\pi(\widetilde{\mathcal U}_{\mathbf f}^{\mathrm{cl}}(r_0)) \subset \mathcal U_{\mathbf f}^{\mathrm{cl}}$. In particular, we may consider the function 
\[
    \mathbf a_{|\d|}(\Theta_{\d}^{r_0}(\mathbf f))\colon \mathcal X(\widetilde{\mathcal R}_{\mathbf f}) \, \longrightarrow \, \C_p, \quad \tilde{\kappa} \, \longmapsto \, \mathbf a_{|\d|}(\Theta_{\d}^{r_0}(\mathbf f))(\tilde{\kappa})
\]
given by the $|\d|$-th Fourier coefficient of $\Theta_{\d}^{r_0}(\mathbf f)$.

\begin{proposition}\label{prop:Lambda-adic-Kohnen}
With notation and assumptions as above, suppose that $\mathrm{gcd}(N,\d)=1$ and $\epsilon(-1)^{r_0}\d > 0$. If $\tilde{\kappa} \in \widetilde{\mathcal U}_{\mathbf f}^{\mathrm{cl}}(r_0)$ has weight $k-1$, $\kappa = \pi(\tilde{\kappa})$, and $\mathbf f(\kappa) \neq f_{\kappa}$, then
\begin{equation}\label{interpolationLp}
    \mathbf a_{|\d|}(\Theta_{\d}^{r_0}(\mathbf f))(\tilde{\kappa}) =  \Omega_{\kappa} \cdot
    \chi_{\d}(-1) R_{\d}(f_{\kappa})(-1)^k |\d|^k N_0^k (k-1)! \cdot \frac{L(f_{\kappa},\chi_{\d}\bar{\chi}_0,k)}{(2\pi i)^k\mathfrak g(\chi_{\d}\bar{\chi}_0)\Omega_{\mathbf f(\kappa)}^-},
\end{equation}
where  $R_{\d}(f_{\kappa})$ as defined in Proposition \ref{prop:rkN-Lvalue-KT}.
\end{proposition}
\begin{proof}
    For any $\tilde{\kappa} \in \widetilde{\mathcal U}_{\mathbf f}^{\mathrm{cl}}(r_0)$ of weight $k-1$, Theorem \ref{thm:LambdaShintani-interpolation} and Corollary \ref{cor:lift-falpha} (see also Remark \ref{rmk:coeff-pstabilization}) imply that
    \begin{align}
        \mathbf a_{|\d|}(\Theta_{\d}^{r_0}(\mathbf f))(\tilde{\kappa}) & =
        \Omega_{\kappa}\cdot \frac{ a_{|\d|}(\theta_{k,Np,\chi,\d}(\mathbf f(\kappa)))}{C(k,\chi,\d)\cdot \Omega_{\mathbf f(\kappa)}^-}
        = \Omega_{\kappa} (1-\beta\chi_{\d}\bar{\chi}_0(p)p^{-k}) \cdot \frac{a_{|\d|}(\theta_{k,Np,\chi,\d}(f_{\kappa}))}{C(k,\chi,\d)\cdot \Omega_{\mathbf f(\kappa)}^-} = \label{interpolation-ad} \\
        & = \Omega_{\kappa} \cdot \frac{a_{|\d|}(\theta_{k,Np,\chi,\d}(f_{\kappa}))}{C(k,\chi,\d)\cdot \Omega_{\mathbf f(\kappa)}^-}, \nonumber
    \end{align}
    where $\kappa = \pi(\tilde{\kappa})$, and the last equality uses that $p$ divides $\d$. If $\mathbf f(\kappa)\ne f_\kappa$, the result follows by combining \eqref{interpolation-ad} with Proposition \ref{prop:thetaN=thetaNp} and Corollary \ref{cor:ad-Lvalue}.
\end{proof}

\begin{remark}
    Notice that the condition $\mathbf f(\kappa) = f_\kappa$ can only happen when $k = 1$.  See the forthcoming Corollary \ref{cor:interpolationLpATone} for the value at those $\kappa$ (in the case that $\chi$ is trivial).
\end{remark}

Note that the values
\[
    L^{\mathrm{alg}}(f_{\kappa},\chi_{\d}\bar{\chi}_0,k) := \frac{L(f_{\kappa},\chi_{\d}\bar{\chi}_0,k)}{(2\pi i)^k\mathfrak g(\chi_{\d}\bar{\chi}_0)\Omega_{\mathbf f(\kappa)}^-}
\]
in \eqref{interpolationLp} are algebraic, hence the Fourier coefficient $\mathbf a_{|\d|}(\Theta_{\d}^{r_0}(\mathbf f))$ interpolates the `algebraic parts' of the special values $L(f_{\kappa},\chi_{\d}\bar{\chi}_0,k)$ at classical points $\tilde{\kappa} \in \widetilde{\mathcal U}_{\mathbf f}(r_0)$. A $p$-adic $L$-function interpolating such special values was studied in Greenberg--Stevens \cite{GS93}, generalizing Mazur--Tate--Teitelbaum \cite{MTT}. Associated with $\mathbf f$ and a Dirichlet character $\psi$, they define a two-variable $p$-adic $L$-function $\mathcal L_p^{\mathrm{GS}}(\mathbf f, \psi)$ on some local domain $\mathcal U_{\mathbf f} \times \mathcal U \subset\mathcal X(\mathcal R_{\mathbf f}) \times \mathcal X(\Lambda)$ satisfying the following interpolation property: for every pair of classical points $(\kappa, j)\in \mathcal U^{\mathrm{cl}}_{\mathbf f}\times \mathcal U^{\mathrm{cl}}$ in the cone defined by $0 < j < \mathrm{wt}(\kappa)+2$, 
\[
    \mathcal L_p^{\mathrm{GS}}(\mathbf f, \psi)(\kappa,j) = \Omega_{\kappa} \cdot \mathcal E_p(\mathbf f(\kappa), \psi,j) \cdot \frac{\mathrm c^{j-1}(j-1)!\mathfrak g(\psi\omega^{1-j})}{(2\pi i)^j\Omega_{\mathbf f(\kappa)}^{\mathrm{sgn}(\psi)}} L(\mathbf f(\kappa),\bar{\psi}\omega^{j-1},j).
\]
Here $c = \mathrm{cond}(\psi\omega^{1-j})$, $m$ is the exponent of $p$ in $\mathrm c$, $a_p(\kappa)$ is the $p$-th Fourier coefficient of $\mathbf f(\kappa)$, and the Euler-like factor:
\begin{equation} \label{eqn:padicmultiplier}
    \mathcal E_p(\mathbf f(\kappa), \psi,j) = a_p(\kappa)^{-m} \left(1-\frac{\psi \omega^{1-j}(p)p^{j-1}}{a_p(\kappa)}\right).
\end{equation}
When restricted to the line $(\kappa, k)$, where $k = (\mathrm{wt}(\kappa) +2)/2$, this identity becomes
\[
    \mathcal L_p^{\mathrm{GS}}(\mathbf f, \psi)(\kappa,k) = \Omega_{\kappa} \cdot \mathcal E_p(\mathbf f(\kappa), \psi,k) \cdot \frac{\mathrm c^{k-1}(k-1)!\mathfrak g(\psi\omega^{1-k})}{(2\pi i)^k\Omega_{\mathbf f(\kappa)}^{\mathrm{sgn}(\psi)}} L(\mathbf f(\kappa),\bar{\psi}\omega^{k-1},k).
\]
This suggests defining a one-variable $p$-adic $L$-function as the restriction of $\mathcal L_p^{\mathrm{GS}}(\mathbf f, \psi)$ to the `line' $(\kappa,k)$. More precisely, this one-variable $p$-adic $L$-function is the pullback of $\mathcal L_p^{\mathrm{GS}}(\mathbf f, \psi)$ along the map 
\[
    \mathcal U_{\mathbf f} \, \stackrel{\Delta}{\longrightarrow} \, \mathcal U_{\mathbf f} \times \mathcal U
\]
given by $\kappa \mapsto (\kappa, \mathrm{wt}(\kappa)/2+1)$ on $\mathcal U_{\mathbf f}^{\mathrm{cl}}$. For our purposes, we will rather consider the pullback of $\mathcal L_p^{\mathrm{GS}}(\mathbf f,\psi)$ along 
\[
    \iota\colon \widetilde{\mathcal U}_{\mathbf f} \, \stackrel{\pi}{\longrightarrow} \, \mathcal U_{\mathbf f} \, \stackrel{\Delta}{\longrightarrow} \, \mathcal U_{\mathbf f} \times \mathcal U,
\]
which yields a one-variable $p$-adic $L$-function $\tilde{\mathcal L}_p^{\mathrm{GS}}(\mathbf f,\psi)$ on $\widetilde{\mathcal U}_{\mathbf f}$. Indeed, by construction we see that if $\tilde{\kappa} \in \widetilde{\mathcal U}_{\mathbf f}^{\mathrm{cl}}(r_0)$ has weight $k-1$ and $\kappa = \pi(\tilde{\kappa})$, then 
\[
    \tilde{\mathcal L}_p^{\mathrm{GS}}(\mathbf f, \psi)(\tilde{\kappa}) =  \Omega_{\kappa} \cdot \mathcal E_p(\mathbf f(\kappa), \psi,k) \cdot \frac{\mathrm c^{k-1}(k-1)!\mathfrak g(\psi\omega^{1-k})}{(2\pi i)^k\Omega_{\mathbf f(\kappa)}^{\mathrm{sgn}(\psi)}} \cdot L(\mathbf f(\kappa),\bar{\psi}\omega^{k-1},k).
\]
Similarly, by pulling back via the map $\pi\colon \widetilde{\mathcal U}_{\mathbf f} \to \mathcal U_{\mathbf f}$ we can define $\Lambda$-adic functions $\mathbf R_\d$ and $\mathbf a_p$ on $\widetilde{\mathcal U}_{\mathbf f}$ such that
\[
    \mathbf R_\d(\tilde\kappa) = R_\d(f_\kappa) \qquad \text{ and } \qquad \mathbf a_p(\tilde\kappa) = a_p(\kappa)
\]
for all $\tilde{\kappa} \in \widetilde{\mathcal U}_{\mathbf f}^{\mathrm{cl}}$ of weight $k-1 > 1$, with $\kappa = \pi(\tilde{\kappa})$. 

The following theorem might be seen as a `$\Lambda$-adic Kohnen formula' in the spirit of the classical formula stated in Corollary \ref{cor:ad-Lvalue}.

\begin{theorem} \label{thm:comparisonGS}
With the above notation, suppose that $\mathrm{gcd}(N,\d)=1$ and $\epsilon(-1)^{r_0}\d > 0$. Then the equality 
\[
    \mathbf a_{|\d|}(\Theta_{\d}^{r_0}(\mathbf f)) = \chi_{\d}(-1) \cdot  \mathbf a_p \cdot \mathbf R_{\d} \cdot \tilde{\mathcal L}_p^{\mathrm{GS}}(\mathbf f, \chi_{\d}\chi_0\omega^{r_0-1})
\]
holds, as an equality of functions on $\widetilde{\mathcal U}_{\mathbf f}$.
\end{theorem}
\begin{proof}
It suffices to prove the claimed equality at classical points in $\widetilde{\mathcal U}_{\mathbf f}^{\mathrm{cl}}(r_0)$ of weight $k-1 > 1$, since they are dense in $\widetilde{\mathcal U}_{\mathbf f}$. Let $\tilde{\kappa} \in \widetilde{\mathcal U}_{\mathbf f}^{\mathrm{cl}}(r_0)$ be a classical point of weight $k-1$, with $k \equiv r_0 \pmod{p-1}$, and let $\kappa = \pi(\tilde{\kappa})$. We consider $\psi = \chi_{\d}\chi_0\omega^{r_0-1}$. Since $k \equiv r_0 \pmod{p-1}$, we have $\psi\omega^{1-k} = \chi_{\d}\chi_0$, which has conductor $|\d|N_0$. Therefore, since $\chi_{\d}(p) = 0$ the Euler-like factor $\mathcal E_p(\mathbf f(\kappa),\psi,k)$ is just $1$. The fact that $\chi_{\d}(p) = 0$ also implies that $L(\mathbf f(\kappa),\bar{\psi}\omega^{k-1},k)=L(\mathbf f(\kappa), \chi_{\d}\bar{\chi}_0,k)=L(f_{\kappa},\chi_{\d}\bar{\chi}_0,k)$. Besides, note that since $\omega$ is odd and we are assuming $\epsilon (-1)^{r_0}\d > 0$ we have $\mathrm{sgn}(\chi_{\d}\chi_0\omega^{r_0-1})=-1$. We thus obtain that 
\[
    \tilde{\mathcal L}_p^{\mathrm{GS}}(\mathbf f, \chi_{\d}\chi_0\omega^{r_0-1})(\tilde{\kappa}) = \Omega_{\kappa} a_p(\kappa)^{-1}\frac{ N_0^{k-1}|\d|^{k-1}(k-1)!\mathfrak g(\chi_{\d}\chi_0)}{(2\pi i)^k\Omega_{\mathbf f(\kappa)}^-} \cdot L(f_\kappa,\chi_{\d}\bar{\chi}_0,k).
\]
Using that
\[
    \mathfrak g(\chi_{\d}\chi_0) = \chi_{\d}\chi_0(-1)|\d|N_0/\mathfrak g(\chi_{\d}\bar{\chi}_0) = (-1)^k|\d|N_0/\mathfrak g(\chi_{\d}\bar{\chi}_0),
\]
we can rewrite the above identity as 
\begin{equation} \label{eqn:GSforus}
    \tilde{\mathcal L}_p^{\mathrm{GS}}(\mathbf f, \chi_{\d}\chi_0\omega^{r_0-1})(\tilde{\kappa}) = \Omega_{\kappa}a_p(\kappa)^{-1} (-1)^kN_0^k|\d|^k(k-1)! \cdot \frac{L(f_\kappa,\chi_{\d}\bar{\chi}_0,k)}{(2\pi i)^k\mathfrak g(\chi_{\d}\bar{\chi}_0)\Omega_{\mathbf f(\kappa)}^-}.
\end{equation}
Since $k-1>1$, we have $\mathbf f(\kappa) \neq f_{\kappa}$ and we can compare the above with equation \eqref{interpolationLp} to deduce that 
\[
    \mathbf a_{|\d|}(\Theta_{\d}^{r_0}(\mathbf f))(\tilde{\kappa}) = a_p(\kappa) \cdot (-1)^k\epsilon \cdot  R_{\d}(f_{\kappa}) \cdot \tilde{\mathcal L}_p^{\mathrm{GS}}(\mathbf f, \chi_{\d}\chi_0\omega^{r_0-1})(\tilde{\kappa}).
\]
To conclude, we may just use that $(-1)^k\epsilon = \chi_{\d}(-1)$ to get
\[
    \mathbf a_{|\d|}(\Theta_{\d}^{r_0}(\mathbf f))(\tilde{\kappa}) = \chi_{\d}(-1) \cdot a_p(\kappa)  \cdot   R_{\d}(f_{\kappa}) \cdot \tilde{\mathcal L}_p^{\mathrm{GS}}(\mathbf f, \chi_{\d}\chi_0\omega^{r_0-1})(\tilde{\kappa}).
\]
\end{proof}

\begin{corollary} \label{cor:interpolationLpATone}
Let $N\geq 1$ be a squarefree integer, $p$ be an odd prime with $p \nmid N$, $f \in S_2(Np)$ be a normalized newform, and $\d < 0$ be a fundamental discriminant divisible by $p$ such that $\mathrm{gcd}(N,\d) = 1$. Assume that $\chi_{\d}(\ell) = w_{\ell}$ for every prime $\ell$ dividing $N$. Let $\mathbf f$ be the Hida family passing through $f$, $\kappa \in \mathcal U_{\mathbf f}^{\mathrm{cl}}$ be such that $\mathbf f(\kappa) = f_{\kappa} = f$, and $\tilde{\kappa} \in \widetilde{\mathcal U}_{\mathbf f}^{\mathrm{cl}}(1)$ such that $\pi(\tilde{\kappa}) = \kappa$. Then 
\[
    \mathbf a_{|\d|}(\Theta_{\d}^1(\mathbf f))(\tilde{\kappa}) = \Omega_{\kappa}  \cdot 2^{\nu(N)} |\d| \cdot \frac{L(f,\chi_{\d},1)}{(2\pi i) \mathfrak g(\chi_{\d})\Omega_f^-}. 
\]
\end{corollary}
\begin{proof}
This is an immedate consequence of the above theorem, taking $k = r_0 = 1$ and noticing that $R_{\d}(f) = 2^{\nu(N)}$ under the hypotheses in the statement.
\end{proof}

By using the interpolation property of the $\Lambda$-adic $\d$-th Shintani lifting, the above corollary can be rewritten in classical terms. This yields a mild generalization of Kohnen's formula in \eqref{ad-Lvalue-basic}, in which the level of the newform and the fundamental discriminant are not relatively prime:

\begin{corollary}\label{cor:KohnenFormula-Np}
Let $N\geq 1$ be a squarefree integer, $p$ be an odd prime with $p \nmid N$, $f \in S_2(Np)$ be a normalized newform, and $\d < 0$ be a fundamental discriminant divisible by $p$ such that $\mathrm{gcd}(N,\d) = 1$. Assume that $\chi_{\d}(\ell) = w_{\ell}$ for every prime $\ell$ dividing $N$. Then 
\[
a_{|\d|}(\theta_{1,Np,\d}(f)) = 2^{1+\nu(N)} |\d| \frac{L(f,\chi_{\d},1)}{(2\pi i) \mathfrak g(\chi_{\d})}.
\]
\end{corollary}

\begin{proof}
    It follows from Corollary \ref{cor:interpolationLpATone}, equation \eqref{interpolation-ad} (where $C(k,\d)=2$) and the fact that the family of $p$-adic periods can be chosen so that the relevant one is non-zero by Theorem \ref{thm:Lambda-adic-MS}.
\end{proof}

\begin{remark}
    This formula could have been obtained by adapting the classical computation of Proposition \ref{prop:rkN-Lvalue-KT}, choosing a suitable set of representatives for $\mathcal L_{Np}(N_0^2\d^2)/\Gamma_0(Np)$ for the computation of $r_{k,Np,\chi}(f;\d,\d)$ (for this, one needs to use arguments similar to those used in Section \ref{Sec:p-stabilizations} to classify integral binary quadratic forms). Instead, the previous corollary shows that such computation can be avoided if one disposes of a $\d$-th Shintani lifting in $p$-adic families. 
\end{remark}

\subsection{Classical points outside the interpolation region} \label{sec:Shadows} 

Continue to assume that $(\mathcal R_{\mathbf f},\mathcal U_{\mathbf f}, \mathcal U_{\mathbf f}^{\mathrm{cl}},\mathbf f)$ is a Hida family as usual, and let $\d$ be a fundamental discriminant as in Theorem \ref{thm:comparisonGS} (depending on a choice of $r_0$). We have seen in the previous paragraph how the $|\d|$-th Fourier coefficient of $\Theta_{\d}^{r_0}(\mathbf f)$ interpolates special values of twisted $L$-series associated with classical specializations of the Hida family $\mathbf f$ on `half' of the classical points. Namely, for classical points $\tilde{\kappa} \in \widetilde{\mathcal U}_{\mathbf f}^{\mathrm{cl}}(r_0)$ of weight $k-1$ the specializations $\mathbf a_{|\d|}(\Theta_{\d}^{r_0}(\mathbf f))(\tilde{\kappa})$ interpolate the special values $L(f_{\pi(\tilde{\kappa})},\chi_{\d}\bar{\chi}_0,k)$ (see Proposition \ref{prop:Lambda-adic-Kohnen}). This was just an immediate consequence of the fact that $\Theta_{\d}^{r_0}(\mathbf f)(\tilde{\kappa})$ interpolates the $\d$-th Shintani liftings of the forms $\mathbf f(\pi(\tilde{\kappa}))$ when varying $\tilde{\kappa} \in \widetilde{\mathcal U}_{\mathbf f}^{\mathrm{cl}}(r_0)$, together with the classical relation between Fourier coefficients of Shintani liftings and special $L$-values.

A natural question is: what about the values of $\mathbf a_{|\d|}(\Theta_{\d}^{r_0}(\mathbf f))$ at classical points $\tilde{\kappa} \in \widetilde{\mathcal U}_{\mathbf f}^{\mathrm{cl}} - \widetilde{\mathcal U}_{\mathbf f}^{\mathrm{cl}}(r_0)$? At those points, we cannot use Theorem \ref{thm:LambdaShintani-interpolation} to relate $\Theta_{\d}^{r_0}(\mathbf f)(\tilde{\kappa})$ to a classical $\d$-th Shintani lifting. However, we can still use Theorem \ref{thm:comparisonGS} to evaluate $\mathbf a_{|\d|}(\Theta_{\d}^{r_0}(\mathbf f))$ by evaluating $\tilde{\mathcal L}_p^{\mathrm{GS}}(\mathbf f,\chi_{\d}\chi_0\omega^{r_0-1})$. Indeed, first notice that if $\tilde{\kappa} \in \widetilde{\mathcal U}_{\mathbf f}^{\mathrm{cl}}$ does not belong to $\widetilde{\mathcal U}_{\mathbf f}^{\mathrm{cl}}(r_0)$, then $\tilde{\kappa}$ has weight $k-1$ for some integer $k$ such that $k \equiv r_0+(p-1)/2$ modulo $p-1$. Then the character 
\[
\chi_{\d}\chi_0\omega^{r_0-1}\omega^{1-k} = \chi_{\d}\omega^{(p-1)/2}\chi_0 = \chi_{\d/p^*}\chi_0
\]
has conductor $\mathrm c = |\d|N_0/p$, where $p^* = (-1)^{(p-1)/2}p$. In particular, $m = 0$ in the notation of the previous paragraph. If $\kappa = \pi(\tilde{\kappa})$, then 
\[
\mathcal E_p(\mathbf f(\kappa), \chi_{\d}\chi_0\omega^{r_0-1},k) = \left(1- \frac{\chi_{\d/p^*}\chi_0(p)p^{k-1}}{a_p(\kappa)}\right),
\]
and therefore we find
\begin{align*}
    \mathbf a_{|\d|}(\Theta_{\d}^{r_0}(\mathbf f))(\tilde{\kappa}) & = a_p(\kappa) \chi_{\d}(-1)\mathbf R_{\d}(\tilde{\kappa})\Omega_{\kappa}\left(1- \frac{\chi_{\d/p^*}\chi_0(p)p^{k-1}}{a_p(\kappa)}\right) \times \\
    & \hspace{3cm} \times \frac{\mathrm c^{k-1}(k-1)!\mathfrak g(\chi_{\d/p^*}\chi_0)}{(2\pi i)^k\Omega_{\mathbf f(\kappa)}^-} \cdot L(\mathbf f(\kappa),\chi_{\d/p^*}\bar{\chi}_0,k).
\end{align*}
Now, we have
\[
    \mathfrak g(\chi_{\d/p^*}\chi_0) = \frac{\chi_{\d/p^*}\chi_0(-1)|\d|N_0}{p\mathfrak g(\chi_{\d/p^*}\bar{\chi}_0)} = \frac{(-1)^{k}|\d|N_0}{p\mathfrak g(\chi_{\d/p^*}\bar{\chi}_0)},
\]
where in the last equality we use that $\chi_{\d/p^*}\chi_0(-1) = (-1)^{(p-1)/2}\chi_{\d}(-1)\epsilon = (-1)^{r_0+(p-1)/2} = (-1)^k$. Hence we can rewrite the above identity as 
\begin{align}\label{interpolationLp-2}
\mathbf a_{|\d|}(\Theta_{\d}^{r_0}(\mathbf f))(\tilde{\kappa}) & = \frac{a_p(\kappa)}{p^k}\left(1 - \frac{\chi_{\d/p^*}\chi_0(p)p^{k-1}}{a_p(\kappa)}\right) \cdot \mathbf R_{\d}(\tilde{\kappa}) \Omega_{\kappa} \cdot \chi_{\d}(-1)    (-1)^k|\d|^kN_0^k(k-1)! \, \, \times \\ 
& \hspace{2cm} \times \frac{L(\mathbf f(\kappa),\chi_{\d/p^*}\bar{\chi}_0,k)}{(2\pi i)^k\mathfrak g(\chi_{\d/p^*}\bar{\chi}_0)\Omega_{\mathbf f(\kappa)}^-}. \nonumber
\end{align}

\begin{remark}
This identity complements the interpolation formula for $\Theta_{\d}^{r_0}(\mathbf f)$ in \eqref{interpolationLp} at classical weights in $\widetilde{\mathcal U}_{\mathbf f}^{\mathrm{cl}} - \widetilde{\mathcal U}_{\mathbf f}^{\mathrm{cl}}(r_0)$. In turn, this identity also suggests that the specializations of $\Theta_{\d}^{r_0}(\mathbf f)$ at classical points in $\widetilde{\mathcal U}_{\mathbf f}^{\mathrm{cl}} - \widetilde{\mathcal U}_{\mathbf f}^{\mathrm{cl}}(r_0)$ would be `$p$-adic shadows' of the $\d$-th Shintani liftings $\theta_{k,Np,\chi\omega^{(p-1)/2},\d}(\mathbf f(\kappa))$. Note that in these $\d$-th Shintani liftings the conductor of the character $\chi\omega^{(p-1)/2}$ is not relatively prime with $\d$, and hence the discussion in Section \ref{Sec:Shintani-Lifting} should be reformulated in order to define such liftings.
\end{remark}

The identity in \eqref{interpolationLp-2} shows also that on the subset of classical points $\widetilde{\mathcal U}_{\mathbf f}^{\mathrm{cl}} - \widetilde{\mathcal U}_{\mathbf f}^{\mathrm{cl}}(r_0)$ one can find {\em exceptional zeroes} (precisely at those points where $\tilde{\mathcal L}_p^{\mathrm{GS}}(\mathbf f,\chi_{\d}\chi_0\omega^{r_0-1})$ has exceptional zeroes). Indeed, it is apparent from \eqref{interpolationLp-2} that for a classical point $\tilde{\kappa} \in \widetilde{\mathcal U}_{\mathbf f}^{\mathrm{cl}} - \widetilde{\mathcal U}_{\mathbf f}^{\mathrm{cl}}(r_0)$ of weight $k-1$, $\mathbf a_{|\d|}(\Theta_{\d}^{r_0}(\mathbf f))(\tilde{\kappa})$ vanishes whenever 
\[
    \chi_{\d/p^*}\chi_0(p)p^{k-1} = a_p(\kappa).
\]
When this holds, the order of vanishing of the $|\d|$-th Fourier coefficient $a_{|\d|}(\Theta_{\d}^{r_0}(\mathbf f))$ at $\tilde{\kappa}$ is at least one more than the order of vanishing of the relevant classical special $L$-value. This extra vanishing is due to the $p$-adic interpolation.

To illustrate one of the settings in which exceptional zeroes arise, suppose that $E/\Q$ is an elliptic curve of conductor $Np$ with multiplicative reduction at the prime $p$ (thus $a_p(E) = \pm 1$ according to whether $E$ has split or non-split multiplicative reduction at $p$, respectively). Let $f \in S_2(Np)$ be the normalized weight $2$ newform (with rational coefficients) associated with $E$ by modularity. Let also $\mathbf f$ be the Hida family passing through $f$ at a classical point $\kappa_0 \in \mathcal U_{\mathbf f}^{\mathrm{cl}}$ (thus $\mathbf f(\kappa_0) = f_{\kappa_0} = f$), and let $\tilde{\kappa}_0 \in \widetilde{\mathcal U}_{\mathbf f}^{\mathrm{cl}}(1)$ be such that $\pi(\tilde{\kappa}_0)=\kappa_0$.

Let $r_0$ be an integer congruent to $1+(p-1)/2 = (p+1)/2$ modulo $p-1$, and $\d$ be a fundamental discriminant divisible by $p$ such that $(-1)^{r_0}\d > 0$. Associated with this choice, consider the $\Lambda$-adic $\d$-th Shintani lifting $\Theta_{\d}^{r_0}(\mathbf f)$.  In order to guarantee the non triviality of the lift, we use Kohnen's formula and consider the assumption that $\chi_{\d}(q) = w_{q^e}(f_\kappa)$ for all primes $q$ such that $q^e \mid\mid N$, for some $\kappa\ne \kappa_0$ in a small enough neighborhood of $\kappa_0$. The quantity $\mathbf R_{\d}(\tilde{\kappa})$ then equals $2^{\nu(N)}$, where $\nu(N)$ is the number of primes dividing $N$, and equation \eqref{interpolationLp-2} reads 
\begin{align*}
    \mathbf a_{|\d|}(\Theta_{\d}^{r_0}(\mathbf f))(\tilde{\kappa}_0) & = (a_p(E)-\chi_{\d/p^*}(p))\cdot  2^{\nu(N)} \Omega_{\kappa_0}\cdot (-\d)N_0p^{-1}\cdot \frac{L(E,\chi_{\d/p^*},1)}{(2\pi i) \mathfrak g(\chi_{\d/p^*})\Omega_f^-},
\end{align*}
and we see that $\mathbf a_{|\d|}(\Theta_{\d}^{r_0}(\mathbf f))$ has an exceptional zero at $\tilde{\kappa}_0$ when $\chi_{\d/p^*}(p)=a_p(E)$.

In this setting, we can combine Theorem 5.4 of \cite{BD07} with the above Theorem \ref{thm:comparisonGS} and find the following:

\begin{corollary}\label{cor:BDderivative}
	With the notation as above, assume that $E$ has at least two distinct primes of semistable reduction and that $\ord_{s=1}L(E,\d/p^*,1) = 1$. Then there exists a global point $P\in (E(\Q[\sqrt{\d/p^*}])\otimes \Q)^{\chi_{\d/p*}}$ and a scalar $\lambda\in \Q^\times$ such that
	\[
		\mathbf a_{|\d|}(\Theta_{\d}^{r_0}(\mathbf f))''(\tilde{\kappa}_0) = \lambda \cdot \log^2(P),
	\]
	where the derivative is taken in the weight direction, in a small enough rigid analytic neighborhood of $\tilde\kappa_0$, and $\log = \log_E$ is the formal group logarithm on $E$. 
\end{corollary}

\begin{proof}
	Under the assumptions, we know that there exists a small enough neighborhood of $\tilde\kappa_0$ such that $\tilde{\mathcal L}_p^{\mathrm{GS}}(\hf,\chi_\d\omega^{r_0-1})$ is nonzero. In particular, for a point $\tilde\kappa\in \widetilde{\mathcal U}_\hf^\mathrm{cl} \setminus \widetilde{\mathcal U}_\hf^\mathrm{cl} (r_0)$ of weight $k-1$ in such a neighborhood we have
	\[
		L(f_\kappa,\chi_{\d/p^*},k) \ne 0.
	\]
	If $\varepsilon(f_\kappa,\chi_{\d/p^*})$ denotes the global sign in the functional equation for this $L$-function, we have $\varepsilon(f_\kappa,\chi_{\d/p^*})= -w_N(f_\kappa)\chi_{\d/p^*}(-N)$. But the above implies that this sign is $+1$, hence we deduce that
	\[
		\chi_{\d/p^*}(-N) = - w_N(f_\kappa).
	\]
	Since $w_N(f_\kappa)$ is constant in the Hida family (in the neighborhood chosen as above) and since we have $\chi_{\d/p^*}(p) = -w_p(f)$, it follows that $\chi_{\d/p^*}(-Np) = w_{Np}(f)$`, and therefore the hypotheses of \cite[Theorem 5.4]{BD07} are satisfied. In particular, part (1) and (3) of loc. cit. implies that $\tilde{\mathcal L}_p^{\mathrm{GS}}(\hf,\chi_\d\omega^{r_0-1})$ has a zero of order 2 at $\tilde\kappa_0$, and from Theorem 5.13 we deduce that
	\[
		\mathbf a_{|\d|}(\Theta_{\d}^{r_0}(\mathbf f))'' = \chi_{\d}(-1) \cdot  \mathbf a_p \cdot 2^{\nu(N)} \cdot \tilde{\mathcal L}_p^{\mathrm{GS}}(\mathbf f, \chi_{\d}\chi_0\omega^{r_0-1})'' + \tilde{\mathcal L}_p^{\mathrm{GS}}(\mathbf f, \chi_{\d}\chi_0\omega^{r_0-1}) \cdot \bigl( \cdots \bigr),
	\]
	which yields
	\[
		\mathbf a_{|\d|}(\Theta_{\d}^{r_0}(\mathbf f))''(\tilde\kappa_0) = - 2^{\nu(N)} \cdot \tilde{\mathcal L}_p^{\mathrm{GS}}(\mathbf f, \chi_{\d}\chi_0\omega^{r_0-1})'' (\tilde\kappa_0).
	\]
	By part (2) of \cite[Theorem 5.4]{BD07}, there exists a point $P\in (E(\Q[\sqrt{\d/p^*}])\otimes \Q)^{\chi_{\d/p*}}$ and a nonzero rational number $\ell$ such that
	\[
		\tilde{\mathcal L}_p^{\mathrm{GS}}(\mathbf f, \chi_{\d}\chi_0\omega^{r_0-1})'' (\tilde\kappa_0) = \ell \cdot \log^2(P),
	\]
	and the result follows since we can chose $\Omega_{\kappa_0}=1$.
\end{proof}

A similar type of result for the first derivative of Fourier coefficients associated with half-integral weight modular forms was obtained in \cite{DarmonTornaria}. Building on the constructions of \cite{BD09}, \cite{BD07} and the results of \cite{Stevens}, the authors found a formula that we might read as
\[
	\mathbf a_{|D|}(\Theta_{\d}^{r_0}(\mathbf f))'(\tilde{\kappa}_0) = \lambda_D \cdot \log(P_D),
\]
for a fundamental discriminant $D$ satisfying appropriate hypotheses and for a specific normalization of the theta lift. It would be interesting to explore the connections with the above corollary.

%======BIBLIOGRAPHY=====

\end{document}